\documentclass{compositio}%[manuscript]{aomart}
\usepackage{amsmath,amssymb,amsthm} 
\usepackage[all]{xy}
\usepackage{extarrows}
\usepackage{graphicx}
\usepackage{bm}
\usepackage{comment, url}
\usepackage[normalem]{ulem}
\usepackage{multicol}

\usepackage[pagewise]{lineno} %\linenumbers 
\newtheorem{thm}{Theorem}[section]
\newtheorem{lem}[thm]{Lemma}
\newtheorem{cor}[thm]{Corollary}
\newtheorem{prop}[thm]{Proposition}  

\theoremstyle{remark}

\theoremstyle{definition}

\newtheorem{rem}[thm]{Remark} 

\newtheorem{eg}[thm]{Examples}

\newtheorem{q}[thm]{Question} 

\newtheorem{def/prop}[thm]{Definition/Proposition}

\numberwithin{equation}{section}

\usepackage[usenames]{color}

\def\Im{\mathop{\mathrm{Im}}\nolimits}
\def\Ker{\mathop{\mathrm{Ker}}\nolimits}
\def\Coker{\mathop{\mathrm{Coker}}\nolimits}
\def\Hom{\mathop{\mathrm{Hom}}\nolimits}

\def\Gal{\mathop{\mathrm{Gal}}\nolimits}

\def\SL{\mathop{\mathrm{SL}}\nolimits}
\def\GL{\mathop{\mathrm{GL}}\nolimits}

\def\mod{\mathop{\mathrm{mod}}\nolimits}

\newcommand{\mf}[1]{{\mathfrak{#1}}}

\newcommand{\bb}[1]{{\mathbb{#1}}}
\newcommand{\mca}[1]{{\mathcal{#1}}}

\newcommand{\To}{\longrightarrow}

\newcommand{\inj}{\hookrightarrow}
\newcommand{\surj}{\twoheadrightarrow}

\newcommand{\congto}{\overset{\cong}{\to}}

\newcommand{\N}{\bb{N}}
\newcommand{\Z}{\bb{Z}}
\newcommand{\Zp}{\bb{Z}_{p}}
\newcommand{\Q}{\bb{Q}}
\newcommand{\Qp}{\bb{Q}_{p}}
\newcommand{\R}{\bb{R}}
\newcommand{\C}{\bb{C}}
\newcommand{\Cp}{\bb{C}_{p}}
\newcommand{\F}{\bb{F}}
\newcommand{\p}{\mf{p}}
\renewcommand{\P}{\mf{P}}
\newcommand{\m}{\mf{m}}

\newcommand{\mh}{\textsc{m}}

\newcommand{\del}{\partial}

\newcommand{\ol}{\overline}

\newcommand{\ds}{\displaystyle}

\newcommand{\wt}[1]{{\widetilde{#1}}}
\newcommand{\wh}[1]{{\widehat{#1}}}

\DeclareMathOperator*{\restprod}%
 {\mathchoice{\ooalign{\ensuremath{\displaystyle\prod}\crcr\ensuremath{\displaystyle\coprod}}}%
             {\ooalign{\ensuremath{\textstyle\prod}\crcr\ensuremath{\textstyle\coprod}}}%
             {\ooalign{\ensuremath{\scriptstyle\prod}\crcr\ensuremath{\scriptstyle\coprod}}}%
             {\ooalign{\ensuremath{\scriptscriptstyle\prod}\crcr\ensuremath{\scriptscriptstyle\coprod}}}%
 }

\newcommand{\pmx}[1]{\begin{pmatrix}#1\end{pmatrix}}
\newcommand{\spmx}[1]{{\small \pmx{#1}}}
\newcommand{\smat}[1]{\bigl(\begin{smallmatrix}#1\end{smallmatrix}\bigr)}

\title[Profinite rigidity for twisted Alexander polynomials] 
{Profinite rigidity for twisted Alexander polynomials} 

\author{Jun Ueki} 
\email{uekijun46@gmail.com}
\address{Department of Mathematics, School of System Design and Technology, Tokyo Denki University\\ 
5 Senju Asahi-cho, Adachi-ku, 120-8551, Tokyo, Japan}

 %\framebox{last updated：\today \ \ \now}}
\begin{document}
	
\begin{abstract} 
We formulate and prove a profinite rigidity theorem for the twisted Alexander polynomials up to several types of finite ambiguity. We also establish torsion growth formulas of the twisted homology groups in a $\Z$-cover of a 3-manifold with use of Mahler measures. We examine several examples associated to Riley's parabolic representations of two-bridge knot groups and give a remark on hyperbolic volumes. 
\end{abstract}

\maketitle

\setcounter{tocdepth}{1}
{\small 
\tableofcontents } 

%\begin{comment}

\section{Introduction} 
Recently the profinite rigidity in 3-dimensional topology is of great interest with rapid progress. Nevertheless, it is still unknown whether there exists a pair $(J,K)$ of distinct \emph{prime} knots with an isomorphism $\wh{\pi}_J \cong \wh{\pi}_K$ on the profinite completions of their knot groups. 
In the previous article \cite{Ueki5} we proved that the Alexander polynomial of a knot $K$ is determined by the isomorphism class of the profinite completion $\wh{\pi}_K$ of the knot group. 
In this article, we formulate a version of profinite rigidity for twisted Alexander polynomials and prove it up to several types of finite ambiguity. 
In addition, we establish asymptotic formulas of the torsion growth of the twisted homology groups in a $\Z$-cover of a 3-manifold, refining the studies of the author \cite{Ueki4} and R.Tange \cite{TangeRyoto2018JKTR}. 
We finally examine several examples associated to Riley's parabolic representations of two bridge knots and give a remark on hyperbolic volumes.

Let $\pi$ be a discrete group of finite type with a surjective homomorphism $\alpha:\pi\surj t^\Z$, where $t$ is a formal element generating $\Im \alpha$. 
Let $O=O_{F,S}$ denote the ring of $S$-integers, where $S$ is a finite set of maximal ideals of the ring of integers $O_F$ of a number field $F$, and let $\rho:\pi\to \GL_N(O)$ be a representation. 
Let $\wt{{\rm Fitt}}H_i(\Ker\alpha, \rho)$ denote the divisorial hull of the Fitting ideal over $O[t^\Z]$ of the $i$-th Alexander module $H_i(\Ker\alpha, \rho)$. 
If $\wt{{\rm Fitt}}H_i(\Ker\alpha, \rho)$ is a principal ideal, then \emph{the $i$-th twisted Alexander polynomial} 
$\Delta_{\rho,i}^\alpha(t)$ is defined as a generator of this ideal, up to multiplication by units of $O[t^\Z]$. 
(We denote by $\dot{=}$ equalities up to this ambiguity. We sometime omit $i$ and $\alpha$ for simplicity.) 
By the universality of profinite completion, this $\rho$ induces a continuous representation 
$\wh{\rho}:\wh{\pi}\to \GL_N(\wh{O})$. Let $\beta:\pi'\surj t^\Z$ and $\tau:\pi'\to \GL_N(O)$ be another such pair. 
We say $(\wh{\alpha},\wh{\rho})$ and $(\wh{\beta},\wh{\tau})$ are \emph{isomorphic} if they admit isomorphisms $\varphi:\wh{\pi}\congto \wh{\pi}'$ and $\psi:t^\wh{\Z}\congto t^\wh{\Z}$ satisfying $\psi\circ \wh{\alpha}= \wh{\beta}\circ \varphi$ and $\wh{\tau}\circ \varphi\sim \wh{\rho}$, where $\sim$ denotes the equivalence by a conjugate in $\GL_N(\wh{O})$. 
Our main question is the following: 

\begin{q} \label{q}
\emph{To what extent} does the isomorphism class of $(\wh{\alpha},\wh{\rho})$ \emph{determine} $\Delta_\rho^\alpha(t)$? 
\end{q} 

We denote by $\ol{\Q}$ the algebraic closure of $\Q$ in $\C$ and assume that a number field $F$ is a subfield of $\ol{\Q}$. 
For each prime number $p$, let $\C_p$ denote the $p$-adic completion of an algebraic closure $\ol{\Q}_p$ of a $p$-adic numbers $\Qp$. 
Such $\C_p$ is known to be algebraically closed. 
We fix an embedding $\iota_p:\ol{\Q}\inj \C_p$, so that the $p$-adic norm is defined for any $z\in \ol{\Q}$. 
Note that the norm map ${\rm Nr}_{F/\Q}:F\to \Q$ induces a map ${\rm Nr}_{F/\Q}:O_{F,S}[t^\Z]\to \Z_S[t^\Z]$, 
where $S$ also denotes the image of $S$ in $\Z$. 
We say that $f(t), g(t)\in \R[t^\Z]$ belong to \emph{the} same \emph{Hillar class} if 
$f(t)\,\dot{=}\,u(t)v(t)$ and $g(t)\,\dot{=}\,u(t)v(t^{-1})$ holds for some $u(t),v(t)\in \R[t^\Z]$. 
Let $A$ be any set. For a $d$-tuple $(\alpha_i)_i\in A^d$, we denote by $[\alpha_i]_i \in A^d/\mca{S}_d$ 
the equivalent class with respect to the natural action of the $d$-th symmetric group $\mca{S}_d$.  
Now our answer to the question is stated as follows: 

\begin{thm} \label{thm} 
Let $\alpha:\pi\surj t^\Z$ and $\rho:\pi\to \GL_N(O)$ over $O=O_{F,S}$ be as above. 
Suppose that $\wt{\rm Fitt}H_i(\Ker \alpha, \rho)$ in $O[t^\Z]$ is a principal ideal 
so that $\Delta_\rho(t)=\Delta_{\rho,i}^\alpha(t)$ is defined. \\[-3mm]

\noindent 
(1) 
The isomorphism class of $(\wh{\alpha},\wh{\rho})$ determines the Hillar class of ${\rm Nr}_{F/\Q}\Delta_{\rho}(t) \in \Z_S[t^\Z]$, 
hence the (Euclidean/$p$-adic) Mahler measure of ${\rm Nr}_{F/\Q}\Delta_{\rho}(t)$. 
If in addition $\Delta_{\rho}(t)$ is reciprocal, then ${\rm Nr}_{F/\Q}\Delta_{\rho}(t)$ itself is determined.\\[-3mm]

\noindent 
(2) Let $p$ be a prime number such that all roots of $\Delta_{\rho}(t)$ are placed on the unit circle in $\C_p$, which is the case for almost all $p$. 
For each root $\alpha_i$ of $\Delta_{\rho}(t)$, let $\zeta_i$ denote the unique root of unity with $|\alpha_i-\zeta_i|_p<1$ and suppose that this $\zeta_i$ is a primitive $l_i$-th root of unity. Put $m:={\rm lcm}\{l_i\}_i>0$, so that we have $p\!\!\not |\,\,l_i$ and $p\!\!\not |\,\,m$. 
Then the isomorphism class of $(\wh{\alpha},\wh{\rho})$ determines the set of such $p$'s and the families $[l_i]_i \in \Z^d/\mca{S}_d$, 
$\{[(\alpha_i/\zeta_i)^w]_i\mid w\in \Z_p^*\}$, and $\{[\zeta_i^u]_i \mid u\in (\Z/m\Z)^*\}$ in $\ol{\Q}^d/\mca{S}_d$. 

Let $s$ be any topological generator of $\Im\wh{\alpha}=t^\wh{\Z}$. 
Then we have $s=t^v$ for a unit $v$ of the Pr\"uffer ring $\wh{\Z}=\varprojlim_n \Z/n\Z$. 
Since the image of $t$ in $\Im \wh{\alpha}$ is not specified, neither is $v$, while $t$ and $v$ are independent of $p$. 
The pair $(u,w)$ is the image of this $v$ under the natural map $\wh{\Z}\surj \Z/m\Z\times \Zp; v\mapsto (u,w)$. 
\end{thm}

If $O$ is a UFD (hence a PID, for $O$ being a Dedekind domain), then $\Delta_\rho(t)$ is defined. 
If $\rho$ is an $\SL_2$-representation or more generally if $\rho$ is conjugate to its dual, then $\Delta_\rho(t)$ is reciprocal (\cite{HillmanSilverWilliams2010}, \cite[Chapter 6.5]{Hillman2}). 
The assertion (1) only tells about each divisor of $\Delta_\rho(t)$ in $O[t^\Z]$, while (2) tells about the multiplicity. 
If a cyclotomic polynomial decomposes in $O[t^\Z]$, then an inevitable ambiguity may occur (Examples \ref{eg.cyclotomic}). 
The proof of (1) is a direct extension of \cite{Ueki4} by Hillar's result \cite[Theorem 1.8]{Hillar2005} on cyclic resultants. 
In the proof of (2), we invoke the character decomposition of the Iwasawa module of a $\Z/m\Z\times \Zp$-cover and some local field theory. 
The following propositions may reduce the ambiguity in (2). 

On explicit detection of $\Delta_\rho^\alpha(t)$, after knowing the degree and eliminating cyclotomic divisors,  
Hillar's method indeed tells that Theorem \ref{thm} (1) is done in finite time with use of Gr\"obner basis, 
while Propositions \ref{propA}, \ref{propB} below tell that so is (2). 

\begin{prop} \label{propA} 
Let $0\neq \alpha \in \ol{\Q}$, let $k$ denote the Galois closure of $\Q(\alpha)$, and put $d=[k:\Q]$. 
Then for any but finite number of prime number $p$ satisfying $d|(p-1)$, we have $\alpha\in \Q_p$, $|\alpha|_p=1$, and $|\alpha^{p-1}-1|_p<1$. 
\end{prop} 

\begin{prop} \label{propB}
Let $\alpha, \beta \in \ol{\Q}_p $ with $|\alpha -1|_p<1$ and $|\beta-1|_p<1$, and suppose that 
there exists some $w\in \Zp^*$ with $\alpha^w=\beta$. 
If $\alpha$ and $\beta$ are conjugate over $\Qp$ and are not roots of unity, then $w$ is a root of unity. 
\end{prop} 

\begin{rem} \label{rem.WDH} 
Even if $\wt{{\rm Fitt}}H_i(\Ker\alpha, \rho)$ is not a principal so that $\Delta_{\rho,i}^\alpha(t)$ is not defined, the ideal ${\rm Nr}_{F/\Q}\wt{{\rm Fitt}}H_i(\Ker\alpha, \rho)$ is still a principal ideal of $\Z_S[t^\Z]$. If we define $\ol{\Delta}{}_{\rho,i}^\alpha(t)$ as its generator, then the same assertion as in Theorem \ref{thm} (1) holds. %(Note that $\wt{\phantom{\bullet}}$ and ${\rm Nr}_{F/\Q}$ commutes.) 

In any case, 
if $\pi$ is of deficiency one (e.g., a knot group) and $H_1(\Ker\alpha, \rho)$ is a torsion $O[t^\Z]$-module, 
then the fractional ideal $\wt{{\rm Fitt}}H_1(\Ker\alpha, \rho)/\wt{{\rm Fitt}}H_0(\Ker\alpha, \rho)$ %in the ideal group of $O[t^\Z]$ 
is a principal ideal, and is generated by Lin--Wada's polynomial $W_\rho^\alpha(t)\in F(t)$  (Proposition \ref{prop.Wada}). 
If $\Delta_{\rho,i}^\alpha(t)$ for $i=0,1$ are defined, then $W_\rho^\alpha(t)\,\,\dot{=}\,\, \Delta_{\rho,1}^\alpha(t)/\Delta_{\rho,0}^\alpha(t)$ up to multiplication by units of $O[t^\Z]$ holds. 
\end{rem}
 
Wada's initial definition of $W_{\rho}(t)$ may be regarded with less indeterminacy, 
defined up to multiplicity by $\pm t^{\pm1}$, and coincides with the Reidemeister torsion $\tau_{\rho}(t)\in F[t]$. 
By Goda's result \cite{Goda2017pja} derived from a deep theory of analytic torsion developed by M\"uller and others, we may conclude the following. 
\begin{cor} \label{cor.vol} 
Let $\rho_{\rm hol}:\pi_K\to \SL_2(O)$ denote the holonomy representation of a hyperbolic knot $K$ over $O=O_{F,S}\subset \C$. 
If %$O=\Z$ or 
$O=O_F$ for an imaginary quadratic field $F$, 
then the hyperbolic volume ${\rm Vol}(K)$ of $K$ is determined by the isomorphism class of $\wh{\rho}_{\rm hol}$. 
\end{cor}

Fox's formula asserts that the absolute cyclic resultants of the Alexander polynomial $\Delta_K(t)$ of a knot $K$ coincide with the sizes of the torsions in the cyclic covers (cf.~\cite{Weber1979}). Hence  
the asymptotic formula of the torsion growth in the $\Z$-cover with use of Mahler measure $\mh$ \cite{GS1991} and its $p$-adic analogue $\mh_p$ \cite{Ueki4} is closely related to the profinite detection of $\Delta_K(t)$.  

Let $\alpha:\pi\surj t^\Z$ and $\rho:\pi\to \GL_N(O)$ be as before, and let $X$ be a connected manifold with $\pi\cong \pi_1(X)$. 
For each $n\in \N_{>0}$, let $X_n\to X$ denote the $\Z/n\Z$-cover in the $\Z$-cover defined by $\alpha$. 
%$X_\infty \to X$ and $X_n\to X$ denote the $\Z$-cover and the $\Z/n\Z$-cover corresponding to $\alpha$ and $\alpha\circ (\Z\surj \Z/n\Z)$ for each $n\in \N_{>0}$. 
Define \emph{the norm polynomial} $\ol{\Delta}_\rho(t)\in \Z[t^\Z]$ to be a generator of the ideal ${\rm Nr}_{F/\Q} \wt{{\rm Fitt}}(H_1(X_\infty, \rho))\cap \Z[t^\Z]$. 
If $\Delta_\rho(t)$ is defined, then we have $\ol{\Delta}_\rho(t)$ $\dot{=}$ ${\rm Nr}_{F/\Q}\Delta_\rho(t)$ in $\Z_S[t^\Z]$.  .

Suppose that $\rho$ is absolutely irreducible and non-abelian. 
Then Tange's result in \cite{TangeRyoto2018JKTR} may be refined as follows (Theorem \ref{thm.mh}), 
where we denote by ${\rm Res}(f(t), g(t)) \in O$ the resultant of $f(t), g(t) \in O[t]$. 

(1) There is a natural isomorphism $\ds H_1(X_\infty,\rho)/(t^n-1)H_1(X_\infty,\rho) \congto H_1(X_n,\rho)$. 

(2) Put $\Psi_n(t)={\rm gcd}(\ol{\Delta}_\rho(t), (t^n-1))$ and $r_n={\rm Res}(\ol{\Delta}_\rho(t), (t^n-1)/\Psi_n(t))$. 
Then for a bounded sequence $c_n=|{\rm tor}H_1(X_\infty,\rho)/\Psi_n(t))|$, 
the equality $|{\rm tor}H_1(X_n,\rho)|=c_n |r_n|\prod_{p\in S}|r_n|_p$ holds. 

(3) The asymptotic formulas 
\[\ds \lim_{n\to \infty} |{\rm tor}H_1(X_n,\rho)|^{1/n}=\mh(\ol{\Delta}_\rho(t))), \ \ \ds \lim_{n\to \infty} ||{\rm tor}H_1(X_n,\rho)||_p^{1/n}=\mh_p(\ol{\Delta}_\rho(t))\]
of torsion growth with use of Mahler measures hold, where $p$ is any prime number.  \\ 

This paper is organized as follows. 
Additional backgrounds and further questions are given in Section 2. 
We prove Theorem \ref{thm} through Sections 3--9. 
In Section 3, we recall that $\rho:\pi\to \GL_N(O)$ induces $\wh{\rho}:\wh{\pi}\to \GL_N(\wh{O})$. 
In Section 4, we define the twisted Alexander polynomial $\Delta_\rho^\alpha(t)$ and prove Proposition \ref{prop.Wada}. 
In Section 5, we recall the continuous homology $H_i(\wh{\Gamma},V_{\wh{\rho}})$ of profinite completions and paraphrase our question. 
In Section 6, we remark some properties of cyclotomic divisors in the profinite group ring $\wh{O}[[t^{\wh{\Z}}]]$, and reduce our question to the cases without cyclotomic divisors. 
In Section 7, we recall Hillar's result on cyclic resultant and obtain a lemma. 
In Section 8, we recall the norm map and prove Theorem \ref{thm} (1). 
In Section 9, we prove Theorem \ref{thm} (2) with use of Iwasawa modules, and give some remarks. 
In Section 10, we prove Propositions \ref{propA} and \ref{propB}, invoking some $p$-adic number theory. 

Section 11 is devoted to refinements of results in \cite{Ueki4} and \cite{TangeRyoto2018JKTR}; 
Fox's formula and torsion growth formulas with use of Mahler measures of the twisted homology groups in a $\Z$-cover of a 3-manifold. 
%$\rho$-twisted analogues of Fox's formula and the asymptotic formulas of torsion growth in a $\Z$-cover with use of Mahler measures. 
We recall backgrounds, state Theorem \ref{thm.mh} with some remarks, and prove them. 
In Section 12, we examine several examples of $\Delta_\rho(t)$ of Riley's parabolic representations of 2-bridge knot groups from viewpoints of Theorems \ref{thm} and \ref{thm.mh}, and give remarks on topological entropies. 
In Section 13, we briefly discuss hyperbolic volume and prove Corollary \ref{cor.vol}. 

\section{Preliminary} 
A question of profinite rigidity is a version of finite detection problem asking to what extent a topological invariant of a 3-manifold $M$ is \emph{determined} by the isomorphism class of the profinite completion $\wh{\pi}$ of the 3-manifold group $\pi=\pi_1(M)$. 
Latest progress due to Wilton--Zalesskii, Wilkes \cite{WiltonZalesskii2017GT, WiltonZalesskii2019CM, Wilkes2018JA, Wilkes2018NZ, BMRS2018} are on the geometries of 3-manifolds, the JSJ-decompositions, graph manifolds, and hyperbolic manifolds of a certain class. 
To be precise, \cite{BMRS2018} proved the profinite rigidity of certain arithmetic Kleinian groups amongst all finitely generated residually finite groups (the absolute profinite rigidity). 
In addition, Liu proved that mapping classes of a compact surface are almost determined by its finite quotient actions, using twisted Reidemeister torsions and twisted Lefschetz zeta functions \cite{YiLiu2019MCarXiv}. 
%Nevertheless, it is still unknown whether there exists a pair $(J,K)$ of distinct prime knots such that the profinite completions of their knot groups admit an isomorphism $\wh{\pi}_J \cong \wh{\pi}_K$. 
The profinite rigidity of knots was asked also by B.~Mazur from a viewpoint of the analogy between knots and prime numbers \cite{Mazur2012}. We refer to \cite{Ueki5}, \cite{BoileauFriedl2015}, and \cite{Reid-ICM2018} for further background. 

We proved in our previous paper \cite{Ueki5} that the isomorphism class of $\wh{\pi}_K$ determines the Alexander polynomial of any knot $K$ in $S^3$. 
Since the image of $\pi_K^{\rm ab}=t^\Z$ in the isomorphism class of $\wh{\pi}_K^{\rm ab}=t^\wh{\Z}$
is not specified, we needed to study the completed Alexander module over the completed group ring $\wh{\Z}[[t^\wh{\Z}]]=\varprojlim_{m, n} \Z/m\Z[t^{\Z/n\Z}]$. 

In this article, we extend this method to the twisted Alexander polynomial $\Delta_{\rho}^{\alpha}(t)$ of a representation $\rho:\pi\to \GL_N(O)$.
We remark that many important representations over $\C$ may be regarded as those over $O=O_{F,S}$ 
(e.g., the holonomy representations of a hyperbolic knot, \cite{CullerShalen1983}). 
The completed ring $\wh{O}$ is sufficiently large, while $\wh{\C}=\{0\}$. Hence we consider representations over $O$. 

We note that Boileau--Friedl \cite{BoileauFriedl2015} studied a profinite rigidity of the twisted Alexander polynomial 
associated to representations over a finite field, in ordered to prove the profinite rigidity of fiberedness of knots. 
They assumed a ``regular isomorphism'' so that the image of $t$ in $\wh{\pi}_K^{\rm ab}$ was assumed to be known. 
We address a slightly difficult problem here. 

The aim of this paper is to investigate to what extent $\Delta_{\rho}^{\alpha}(t)$ is (abstractly) 
recovered from the isomorphism class of the continuous representation $\wh{\rho}:\wh{\pi}\to \GL_N(\wh{O})$. 
Namely, suppose that $\pi$ runs through all discrete groups with $\alpha:\pi\surj t^\Z$. 
Let $\mca{R}$ denote the set of the pairs $(\rho, \alpha)$ of the representations $\rho:\pi\to \GL_N(O)$ and $\alpha$, 
$\mca{P}$ the set of twisted Alexander polynomials $\Delta_{\rho}^\alpha(t)$, 
and $\wh{\mca{R}}$ the set of isomorphism class of pairs $(\wh{\rho}, \wh{\alpha})$ of 
the profinite completions $\wh{\rho}:\wh{\pi}\to \GL_N(\wh{O})$ and $\wh{\alpha}:\wh{\pi}\surj t^{\wh{\Z}}$.  
Then the question is to what extent the correspondence $\mca{F}:\mca{R}\to \mca{P}$ factors through 
the profinite completion \ $\wh{}\ :\mca{R}\to \wh{\mca{R}}; (\rho,\alpha) \mapsto \text{the isomorphism class of }(\wh{\rho},\wh{\alpha})$. 

%We remark that Liu \cite{YiLiu2019MCarXiv} also studied 

Here we attach further questions, which may be addressed in a future study. 
\begin{q} \label{furtherq} 
(1) 
Recall that the fiberedness of a 3-manifold is a profinite property (\cite[Theorem 1.2]{BoileauFriedl2015}, \cite{BridsonReidWilton2017}, \cite{JaikinZapirain2017fibering}). 
Let $\rho_n$ denote the $n$-th higher holonomy representation for each $n>2$ of a hyperbolic knot $K$ (cf. Section \ref{sec.vol}). 
Then  $W_{\rho_n}(t)$'s are reciprocal. 
If $K$ is fibered, then $\tau_{\rho_n}(t)$'s are monic, while the converse is conjectured \cite{DunfieldFriedlJackson2012, Porti2018ALM}. 
If $\tau_{\rho_n(t)}$'s are monic, then so are $W_{\rho_n}(t)$'s. Is ${\rm Vol}(K)$ determined in this case? 

(2) To what extent the Reidemeister torsion $\tau_\rho(t)=\tau_{\rho\otimes \alpha}(S^3-K)$ is determined by the isomorphism class of $\wh{\rho}$? Does the argument for (1) extends to $\tau_\rho(t)$? 

(3) Does the set of all continuous representations of $\wh{\pi}$ determines ${\rm Vol}(K)$? 
When we study topology of knots by using twisted Alexander polynomials, difficulty lies in how to find a representation $\rho$ conveying topological information. In order to apply our theorems to a more rustic question of profinite rigidity, we need to combine another type of rigidity theorem finding a good $\rho$ among the set of all representations, such as mentioned by Francaviglia in \cite{Francaviglia2004IMRN}. 
Indeed, ${\rm Vol}(K)$ is the maximal value of volumes if $\rho$ runs through all geometric representation in a neighborhood of $\rho_{\rm hol}$. 
We may need an additional insight to obtain a similar nature over all continuous representation, or to detect a nice class of representations. 
\end{q}

\section{Continuous representations} %$\wh{\rho}:\wh{\pi}\to \GL_N(\wh{O})$} 

Let $\pi$ be a discrete group. Then the set of finite quotients $\{\pi/\varpi \mid \varpi\lhd \pi \text{\ of finite index}\}$ forms a directed inverse system with respect to the quotient maps. 
\emph{The profinite completion} $\wh{\pi}$ of $\pi$ is the inverse limit $\varprojlim_{\varpi \lhd \pi}\pi/\varpi$ as a group and endowed the weakest topology such that $\Ker(\wh{\pi}\surj \pi/\varpi)$ is an open subgroup for every normal subgroup $\varpi\lhd \pi$ of finite index. 
It is known that such an inverse system may be reconstructed from the set of isomorphism classes of finite quotients of $\pi$ \cite[Corollary 3.2.8]{RibesZalesskii2010}. 
A discrete group $\pi$ is said to be \emph{residually finite} if for any $g\in \pi$ there exists some normal subgroup $\varpi\lhd \pi$ of finite index such that the image of $g$ in the quotient $\pi/\varpi$ is non-trivial. 
It is equivalent to that the natural map $\iota: \pi\to \wh{\pi}$ is an injection. 
For any oriented connected closed 3-manifold $M$, the group $\pi=\pi_1(M)$ is residually finite (\cite{Hempel1987}+\cite{PerelmanGC1, PerelmanGC2, PerelmanGC3}). 

For a number field $F$, let $O_F$ denote the ring of integers of $F$ and let $S$ be a finite set of prime ideals of $O_F$. 
The ring $O_{F,S}$ of $S$-integers of $F$ is a Dedekind domain obtained by adding all inverse elements of primes in $S$ to $O_F$. For any $F$, there is some $S$ such that $O_{F,S}$ is a UFD, hence a PID. 
The profinite completion $\wh{O}$ of the ring $O=O_{F,S}$ is a topological ring which is defined as the inverse limit $\varprojlim_I O/I$ of the set of finite quotients $O/I$ as a ring and endowed with the profinite topology. 
We have an inverse system $\{\GL_N(O/I)\}_I$ of finite groups and an isomorphism of profinite groups $\GL_N(\wh{O})\cong \varprojlim_I \GL_N(O/I)$. 
%\\ 

Now let $\rho:\pi\to \GL_N(O)$ be a representation of a discrete group. 
The natural map $O\inj \wh{O}$ induces a representation $\rho:\pi\to \GL_N(\wh{O})$. 
By the universality of the profinite completion \cite[Lemma 3.2.1]{RibesZalesskii2010}, for any continuous homomorphism $f:\pi\to H$ to a profinite group, there is a unique continuous homomorphism $\wh{f}:\wh{\pi}\to H$ such that $f$ coincides with the composite $\wh{f}\circ \iota$ of the natural map $\iota: \pi\to \wh{\pi}$ and $\wh{f}$. 
Hence we have an induced continuous representation $\wh{\rho}:\wh{\pi}\to \GL_N(\wh{O})$. 
In other words, we have a natural map \[\wh{}\ :\Hom(\pi, \GL_N(O))\to \Hom(\wh{\pi}, \GL_N(\wh{O})).\] 
If $\pi$ is a residually finite group, then this map is an injection. 

\section{Twisted Alexander polynomials}  
%$\Delta_{\rho}(t)$}
We recall in Subsection \ref{subsec.twisted} the twisted Alexander polynomial and prepare for the proof of Theorem \ref{thm}. 
We give in Subsection \ref{subsec.Wada} additional information related to Remark \ref{rem.WDH} on Wada's invariant, and prove Proposition \ref{prop.Wada}. 

%\subsection{Twisted Alexander modules} \label{subsec.H} 

\subsection{Twisted Alexander polynomials} \label{subsec.twisted}
%Divisorial hull and Wada's invariant} 

Let $M$ be a finitely generated module over a commutative ring $R$. 
If $R^m\overset{A}{\To} R^n\surj M\to 0$ is a presentation of $M$, then 
\emph{the $i$-th Fitting ideal} ${\rm Fitt}_i M$ of $M$ over $R$ is defined as 
the ideal of $R$ generated by all $(n-i)$-minors of $A\in M_{n,m}(R)$, 
which is known to be independent of the choices of a presentation. 
If $r$ is the lowest $i$ with ${\rm Fitt}_i M\neq 0$, then we simply write ${\rm Fitt} M={\rm Fitt}_r M$ and call it \emph{the Fitting ideal} of $M$. 

\emph{The divisorial hull} (or \emph{the reflexive hull}) $\wt{a}$ of an ideal $\mf{a}$ of a ring $R$ is the intersection of all principal ideals containing $\mf{a}$. 
If $R$ is an integrally closed Noether domain (e.g., $R=O[t^\Z]$) and $\mf{a}\neq 0$, then we have $\mf{a}=\cap_\p \mf{a}_{(\mf{p})}$, where $\p$ runs through all prime ideals of height one and $\mf{a}_{(\p)}$ denotes the localization at $\mf{p}$ \cite[Lemma 3.2]{Hillman2}. 
If $R$ is a UFD, then  $\wt{\mf{a}}$ is a principal ideal, generated by the highest common factor of the elements of $\mf{a}$.

Now let $\pi$ be a discrete group of finite type with a surjective homomorphism $\alpha:\pi\surj t^\Z$ 
and a representation $\rho:\pi\to \GL_N(O)$ over $O=O_{F,S}$. 
Let $V_{\rho}$ denote the module $O^n$ regarded as a right $\pi$-module via the transpose of $\rho$. 
Then $\Gamma:=\Ker \alpha$ acts on $V_{\rho}$ via the restriction of $\rho$. 

Let $\mca{F}$ be a projective resolution of $\Z$ over $\Z[\Gamma]$. Then the homology $H_i(\Gamma,V_\rho)$  with local coefficients is defined to be the homology of the complex $\mca{F}\otimes_{\Gamma} V_\rho$. 
We write $H_i(\Gamma,V_\rho)=H_i(\Gamma,\rho)$ for simplicity. 
The conjugate action of $\pi$ on $\Gamma$ induces the action of $\Im (\alpha)=t^\Z\cong \pi/\Gamma$ on $H_i(\Gamma,\rho)$ \cite[Chapter III, Corollary 8.2]{Brown}. 

Let $X$ be an Eilenberg--MacLane space $K(\pi,1)$ of $\pi$ (e.g., the knot exterior if $\pi$ is a knot group), 
and let $X_\infty\to X$ denote the $\Z$-cover corresponding to $\Gamma$. 
Then we have a natural isomorphism $H_i(\Gamma, \rho)\cong H_i(X_\infty, \rho)$ of finitely generated $O[t^\Z]$-modules for $i=0,1$. 
We also note that Shapiro's lemma \cite[III, Proposition 6.2]{Brown} yields $H_i(\Gamma, \rho)\cong H_i(\pi, \rho\otimes \alpha)$ and $H_i(X_\infty, \rho)\cong H_i(X, \rho\otimes \alpha)$ for the tensor representation $\rho\otimes \alpha:\pi\to \GL_N(O[t^\Z])$.

If $\wt{\rm Fitt}H_i(\Gamma, \rho)$ is a principal ideal, then \emph{the $i$-th twisted Alexander polynomial} $\Delta_{\rho,i}^\alpha(t)$ is defined as a generator $\Delta_{\rho,i}^\alpha(t)$ of this ideal, which is well-defined up to $\dot{=}$ in $O[t^\Z]$. 
This $\Delta_{\rho,i}^\alpha(t)$ is known to be an invariant of the isomorphism class of $(\rho,\alpha)$. We sometime omit $i$ and $\alpha$.

\subsection{Lin--Wada's invariant $W_\rho^\alpha(t)$} \label{subsec.Wada} 
Lin--Wada's invariant $W_{\rho}^\alpha(t) \in F(t)$ was initially introduced by Lin \cite{Lin2001AMS} and Wada \cite{Wada1994}. 
It may be defined as the ratio of generators of $\wt{\rm Fitt}$ of \emph{based Alexander modules}, 
hence defined up to $\dot{=}$ over $O[t^\Z]$. 
If $\pi$ is of deficiency one (e.g., a knot group), then $W_\rho^\alpha(t)$ turns out to be a polynomial in $F[t^\Z]$. 
If a presentation of $\pi$ is given, then $W_{\rho}^\alpha(t)$ is explicitly calculated by using Fox differential 
\cite{Wada1994, KirkLivingston1999T1, DunfieldFriedlJackson2012, FriedlVidussi2011survey}. 
If $\Delta_{\rho,i}^\alpha(t)$ are defined for $i=0,1$, then we have $W_{\rho}^\alpha(t)=\Delta_{\rho,1}^\alpha(t)/\Delta_{\rho,0}^\alpha(t)$. 
If $O$ is not a UFD, then $\wt{{\rm Fitt}}H_i(\Gamma, \rho)$ is not necessarily a principal ideal, so that $\Delta_{\rho,i}^\alpha(t)$ may not be defined. Nevertheless, we have the following. 

\begin{prop} \label{prop.Wada} 
Suppose that $\pi$ is of deficiency one (e.g., a knot group). 
Then the fractional ideal $\wt{{\rm Fitt}}H_1(\Gamma,\rho)/\wt{{\rm Fitt}}H_0(\Gamma,\rho)$ in the ideal group of $O[t^\Z]$ is a principal ideal generated by Wada's polynomial $W_\rho^\alpha(t)$. 
\end{prop} 

Hillman's theorem \cite[Theorem 3.12 (3)]{Hillman2} immediately extends to the following lemma. 

\begin{lem} \label{Hillman} 
Let $0\to K\to M\to C\to0$ be an exact sequence of modules over a Noether ring $R$ and suppose $r={\rm rank}(C)$ and $s={\rm rank}(K)$. 
If $K$ is a torsion module, then $\wt{{\rm Fitt}}_{r+s}(M)=\wt{{\rm Fitt}}_s(K)\wt{{\rm Fitt}}_r(C)$ holds.
\end{lem} 
\begin{proof}
Since an ideals in a Noether ring $R$ is determined by localization at every maximal ideals, we may assume that $R$ is a local ring, so that every projective $R$-module is free. 
In order to prove the claim for the divisorial hull, it suffices to show ${\rm Fitt}_{r+s}(M)={\rm Fitt}_s(K){\rm Fitt}_r(C)$ for the localization at every prime ideal of hight one. 
If the projective dimension of $K$ and $C$ are less than one, we have 
presentation matrixes $P(K)\in {\rm M}_{k+s,k}(R)$ of rank $s$ for $K$ and $P(C)\in {\rm M}_{c+r,c}(R)$ of rank $r$ for $C$, 
hence $P(M)=\spmx{P(K)&0\\ \ast&P(C)}$ of rank $r+s$ for $M$. This proves the equality for the localization, hence the claim. \end{proof}

\begin{proof}[Proof of Proposition \ref{prop.Wada}] 
By the proof of \cite[Proposition 3.6]{SilverWilliams2009TA} (see also \cite[Theorem 4.1]{KirkLivingston1999T1}, \cite[Theorem 11]{HillmanLivingstonNaik2006}), we have an exact sequence 
$$0\to H_1(\Gamma,\rho)\to \mca{A}\to \mca{C}\to H_0(\Gamma,\rho)\to 0$$
of torsion $O[t^\Z]$-modules for \emph{based Alexander modules} $\mca{A}$ and $\mca{C}$.
The Fitting ideals of $\mca{A}$ and $\mca{C}$ are known to be principal ideals. 
Indeed, this $\mca{A}$ admits a square presentation and ${\rm Fitt}(\mca{A})$ is a principal ideal generated by the numerator of Wada's invariant.

By using Lemma \ref{Hillman} twice, we obtain $\wt{{\rm Fitt}}H_1(\Gamma,\rho)\cdot {\rm Fitt}\mca{C}={\rm Fitt}\mca{A}\cdot \wt{{\rm Fitt}}H_0(\Gamma,\rho)$ of ideals in $O[t^\Z]$. 
Therefore, we have the equality of fractional ideals $\wt{{\rm Fitt}}H_1(\Gamma,\rho)/\wt{{\rm Fitt}}H_0(\Gamma,\rho)={\rm Fitt}\mca{A}/{\rm Fitt}\mca{C}=(W_\rho(t))$. \end{proof} 

Wada's invariant is defined for any discrete group $\pi$ of finite type. 
Even if we do not assume that ${\rm Fitt}\mca{A}$ and ${\rm Fitt}\mca{C}$ are principal ideals, we still have the equality 
$\wt{{\rm Fitt}}H_1(\Gamma,\rho)/\wt{{\rm Fitt}}H_0(\Gamma,\rho)=\wt{\rm Fitt}\mca{A}/\wt{\rm Fitt}\mca{C}$.  

\section{Continuous homology and profinite completions} 
In this section, we recall the continuous homology $H_i(\wh{\Gamma}, V_{\wh{\rho}})$ and paraphrase our question.

Let $\pi$, $\alpha:\pi\surj t^\Z$, and $\rho:\pi \to \GL_N(O)$ be as in Subsection \ref{subsec.twisted}. %
Let $V_{\wh{\rho}}$ denote the module $\wh{O}^n$ regarded as a right $\wh{\pi}$-module $\wh{O}^n$ via the transpose of the continuous representation $\wh{\rho}:\wh{\pi}\to \GL_N(\wh{O})$. 
The homology of the profinite group $\wh{\Gamma}$ with coefficient being the profinite module $V_\wh{\rho}$ is defined as the continuous homology of the complex $\wh{\mca{F}}\wh{\otimes}_{\wh{\Gamma}}V_{\wh{\rho}}$, where $\wh{\mca{F}}$ is a  continuous projective resolution of $\wh{\Z}$ over the profinite group ring $\wh{\Z}[[\wh{\Gamma}]]$. 
If $G$ is a finite discrete group and $B$ a finite ring, then the continuous homology coincides with the usual group homology. 

Let $J$ be an ordered set, $(\pi_j)_{j \in J}$ an inverse system of profinite groups, and $(B_j)_{j \in J}$ an inverse system of profinite abelian groups. If $\pi=\varprojlim \pi_j$ and $B=\varprojlim_j B_j$, then for each $i\geq 0$, we have an isomorphism $H_i(\pi,B)\cong \varprojlim_j H_i(\pi_j, B_j)$ \cite[Proposition 6.5.7]{RibesZalesskii2010}. 

Since the set of normal subgroups $G$ of $\Gamma$ of finite index and the set of ideals $I$ of $O$ of finite index are ordered sets with respect to the inclusions, by using the assertion above twice, we obtain isomorphisms 
\[\ds H_i(\wh{\Gamma}, V_{\wh{\rho}})\cong \varprojlim_{G\lhd \Gamma} H_i(\Gamma/G, V_{\wh{\rho}})\cong \varprojlim_{G\lhd \Gamma} \varprojlim_{I} H_i(\Gamma/G, V_{\rho \mod I}).\] 
Since every $H_i(\Gamma/G, V_{\rho \mod I})$ is a quotient of a common finitely generated $O[t^\Z]$-module $H_i(\Gamma, \rho)$, the module $H_i(\wh{\Gamma}, V_{\wh{\rho}})$ is a finitely generated $\wh{O}[[t^\wh{\Z}]]$-module, where $\wh{O}[[t^\wh{\Z}]]$ denotes the profinite completion of the group ring $O[t^\Z]$.\\

Now recall that our aim is to reconstruct $\Delta_\rho(t)=\Delta_{\rho,i}^\alpha(t)$ from the isomorphism class of $(\wh{\rho}, \wh{\alpha}$). 
For the induced map $\wh{\alpha}:\wh{\pi}\surj t^\wh{\Z}$ and $\Gamma=\Ker \alpha$, we have $\Ker \wh{\alpha}=\wh{\Gamma}$. 
The closure of $\wt{{\rm Fitt}}H_i(\Gamma, \rho)$ in $\wh{O}[[t^{\wh{\Z}}]]$ coincides with the ideal $\wt{{\rm Fitt}}H_i(\wh{\Gamma}, \wh{\rho})$ of the continuous homology. 
Let $s$ be an arbitrary generator of $t^{\wh{\Z}}$. Then there exists some $v\in \wh{\Z}^*$ with $t=s^v$. 
Since $\wh{O}[[t^\wh{\Z}]]=\wh{O}[[s^\wh{\Z}]]$, the module $H_i(\wh{\Gamma}, V_{\wh{\rho}})$ is a finitely generated $\wh{O}[[s^\wh{\Z}]]$-module with $\wt{\rm Fitt}$ being $(\Delta_\rho(s^v))$, supposing that $\Delta_\rho(t)$ is defined.  

Hence we ask to what extent $\Delta_\rho(t)\,\dot{=}\,g(t)$ holds for $g(t)\in O[t^\Z]$ and $v'\in \wh{\Z}^*$ with $(\Delta_\rho(s^v))=(g(s^{v'}))$. 
By replacing $s^v$ and $v'/v$ by $s$ and $v$, we may paraphrase our question as follows. 

\begin{q} \label{question} Let $f(t), g(t) \in O[t^\Z]$ and $v\in \wh{\Z}^*$ with an equality $(f(t))=(g(t^v))$ of ideals in $\wh{O}[[t^\wh{\Z}]]$. To what extent does $f(t)\,\dot{=}\,g(t)$ hold? 
\end{q} 

\section{Cyclotomic divisors} \label{Sec.Cyclotomic}
For each positive integer $m$, the $m$-th cyclotomic polynomial $\Phi_m(t)$ is the minimal polynomial of a primitive $m$-th root of unity over $\Q$. It is known that $\Phi_m(t)\in \Z[t]$, that $\Phi_m(t)$ vanishes at every primitive $m$-th root of unity, and that $\prod_{m|n} \Phi_m(t)=t^n-1$ holds. 

By a similar argument to \cite[Section 3]{Ueki5}, we obtain the following properties. 

\begin{prop} \label{prop.cyclotomic}
Let $O=O_{F,S}$ be the ring of $S$-integers of a number field $F$ and regard $O[t^\Z]\subset \wh{O}[[t^{\wh{\Z}}]]$. 

{\rm (1)} Any polynomial $0\neq f(t) \in O[t^{\Z}]$ is not a zero divisor of $\wh{O}[[t^\wh{\Z}]]$.

{\rm (2)} For any cyclotomic polynomial $\Phi_m(t)$ and a unit $v\in \wh{\Z}^*$, the ratio $\Phi_m(t^v)/\Phi_m(t)$ is defined and is a unit of $\wh{O}[[t^\wh{\Z}]]$. 

{\rm (3)} Let $f(t), g(t)\in O[t^\Z]$ and $v\in \wh{\Z}^*$ with an equality $(f(t))=(g(t^v))$ of ideals of $\wh{O}[[t^\wh{\Z}]]$. Then $\Phi_m(t)|f(t)$ holds if and only if $\Phi_m(t)|g(t)$ holds, and these conditions imply the equality $(f(t)/\Phi_m(t))=(g(t^v)/\Phi_m(t^v))$ of ideals in $\wh{O}[[t^{\wh{\Z}}]]$. 
\end{prop} 

By these assertions, the proof of Theorem \ref{thm} (1) and Question \ref{question} reduce to the cases without cyclotomic divisors. 
Note that if $\Phi_m(t)$ decomposes in $O[t]$, then a divisor of $\Phi_m(t)$ may appear as a divisor of $f(t)$. 
Theorem \ref{thm} (1) cannot distinguish divisors of $\Phi_m(t)$, while (2) may tell more information about the multiplicity. 

\begin{eg} \label{eg.cyclotomic}
Suppose that $\sin 2\pi/5 \in O \subset \C$, and let $\zeta=\zeta_5$ be a 5th root of unity. Then we have a decomposition $\Phi_5(t)=\phi_5^+(t)\phi_5^-(t)$ for $\phi_5^+(t)=(t-\zeta)(t-\zeta^4)$ and $\phi_5^-(t)=(t-\zeta^2)(t-\zeta^3)$ in $O[t]$.
Since the images of $\phi_5^+(t)$ and $\phi_5^-(t)$ under $O[t^\Z]\surj O[t^\Z]/(t-\zeta)$ do not coincide, $\phi_5^+(t)\,\dot{=}\,\phi_5^-(t)$ does not hold. 
However, we have the equality $(\phi_5^+(t))=(\phi_5^-(t^v))$ of ideals in $\wh{O}[[t^{\wh{\Z}}]]$ for some unit $v\in \wh{\Z}$.
Hence we cannot distinguish these polynomials. 
\end{eg} 

If such an ambiguity comes from an automorphism of $\pi$, then it is not essential for our Question \ref{question}, while multiplicity of divisors should be cared. See also Examples \ref{eg.Riley} (iii). 

\section{Cyclic resultants and Hillar's theorem} \label{Sec.Hillar} 

For each positive integer $n$, 
the $n$-th cyclic resultant $r_n$ of a polynomial $f(t)\in \C[t]$ is defined as the resultant of $f(t)$ and $t^n-1$. We have $r_n=\prod_{\zeta^n=1} f(\zeta)$. 
If $f(t)\in O[t]$, then we have $r_n\in O$. We call $|r_n|$ the $n$-th cyclic resultant absolute value of $f(t)$. (We consult \cite{Ueki4} for the general definition with more detailed properties.) 

The following result due to Hillar is a generalization of Fried's deep result \cite[Proposition]{Fried1988}: 
\begin{prop}{\rm \cite[Theorem 1.8]{Hillar2005}.} \label{prop.Hillar}
Polynomials $f(t), g(t)\in \R[t]$ have the same sequence of non-zero cyclic resultant absolute values if and only if there exists $u(t), v(t)\in \C[t]$ with $u(0)\neq 0$ and $l_1,l_2\in \Z_{\geq0}$ satisfying $f(t)=\pm t^{l_1}v(t)u(t^{-1})t^{\deg (u)}$, and $g(t)=t^{l_2}v(t)u(t)$. 
\end{prop}

In this article we say that $f(t), g(t)\in \R[t^\Z]$ belong to \emph{the} same \emph{Hillar class} if there exist some $u(t),v(t)\in \R[t^\Z]$ satisfying $f(t)\,\dot{=}\,u(t)v(t)$ and $g(t)\,\dot{=}\,u(t)v(t^{-1})$ in $\R[t^\Z]$. 
We indeed have $u(t), v(t)\in \R[t^\Z]$ in Proposition \ref{prop.Hillar}, hence the following. 
\begin{cor} If $f(t),g(t)\in \Z_S[t^\Z]$ have the same sequence of non-zero cyclic resultant absolute values, 
then there are some  $u(t), v(t)\in \R[t^\Z]$ with $f(t)\,\dot{=}\,u(t)v(t)$ and $g(t)\,\dot{=}\,u(t)v(t^{-1})$, namely, $f(t),g(t)$ belong to the same Hillar class. 
\end{cor} 

The following assertion is obtained in the proof of \cite[Lemma 3.6]{Ueki4}. 
\begin{prop} Let $f(t), g(t)\in \Z[t]$, $v\in \wh{\Z}^*$ and suppose that they have no root on roots of unity. 
If the equality $(f(t))=(g(t^v))$ of ideals in $\wh{\Z}[[t^{\wh{\Z}}]]$ holds, then they have the same sequences of  non-zero cyclic resultant absolute values. 
\end{prop} 

Combining these above, we obtain the following lemma. 
\begin{lem} \label{lem.rn} 
Let $f(t), g(t)\in \Z[t]$ with no roots on roots of unity, and let $v\in \wh{\Z}^*$. 
Then the equality $(f(t^v))=(g(t))$ of ideals in $\wh{\Z}[[t^{\wh{\Z}}]]$ implies that $f(t)$ and $g(t)$ belong to the same Hillar class. 
\end{lem} 

\section{Norm maps} \label{Sec.Norm} 
The norm map ${\rm Nr}_{F/\Q}:F\to \Q$ of a finite extension $F/\Q$ is defined by $x\mapsto \prod_\sigma \sigma(x)$ where $\sigma$ runs through all embeddings $\sigma:F \inj \ol{\Q}$. If $x\in \Q$, then ${\rm Nr}_{F/\Q} x=x^{[F:\Q]}$ holds. 
For a finite set $S$ of maximal ideals of the ring $O_F$ of integers of $F$, let $S$ also denote the set of prime numbers of $\Z$ under $S$. Then the norm map restricts to the norm map ${\rm Nr}_{F/\Q}:O_{F,S}\to \Z_S$ on the $S$-integers. 
In addition, for a Laurent polynomial $f(t)=\sum_i a_i t^i\in O_{F,S}[t^\Z]$ and an embedding $\sigma:F \inj \ol{\Q}$, put $f^{\sigma}(t)=\sum_i a_i^\sigma t^i \in O_{F,S}[t^\Z]$. Then the map ${\rm Nr}_{F/\Q}:O_{F,S}[t^\Z]\to \Z_S[t^\Z]; f(t)\mapsto \prod_\sigma f^\sigma(t)$ is defined. 
This map coincides with the restriction of the norm map of a separable extension of a field of rational functions. 
Furthermore, since the preimage of each $(n,1-t^m)\subset \Z_S[t^\Z]$ by ${\rm Nr}_{F/\Q}$ contains $(n,1-t^m) \subset O_{F,S}[t^\Z]$, 
we have an induced map ${\rm Nr}_{F/\Q}: \wh{O}_{F,S}[[t^{\wh{\Z}}]]\to \wh{\Z}_S[[t^{\wh{\Z}}]]$ on the profinite completions. 

Let $\mf{a}$ be an ideal of $O_{F,S}[t^\Z]$. 
If $\mf{a}$ is a principal ideal, then ${\rm Nr}_{F/\Q}\mf{a}$ is naturally defined. 
For a general $\mf{a}$, we define ${\rm Nr}_{F/\Q}\mf{a}$ to be the ideal generated by $\{{\rm Nr}_{F/\Q}a\mid a\in \mf{a}\}$. It is equivalent to put ${\rm Nr}_{F/\Q}\mf{a}:=(\prod_\sigma \mf{a}^\sigma)\cap O_{F,S}[t^\Z]$. 
This correspondence also extends to the ideals of $\wh{O}_{F,S}[[t^\wh{\Z}]]$ and $\wh{\Z}_S[[t^\wh{\Z}]]$. 

\begin{proof}[Proof of Theorem \ref{thm} (1)]
Theorem \ref{thm} is an answer to the question asked in Section 1. The question is paraphrased to Question \ref{question}, and reduces to the cases without cyclotomic divisors by Proposition \ref{prop.cyclotomic}. 

Recall $O=O_{F,S}$. 
Let $f(t), g(t) \in O[t^\Z]$ and $v\in \wh{\Z}^*$ with the equality $(f(t))=(g(t^v))$ of ideals of $\wh{O}[[t^{\wh{\Z}}]]$. By Proposition \ref{prop.cyclotomic}, we may assume that they have no root on roots of unity. 
Put $F(t):={\rm Nr}_{F/\Q}f(t)$, $G(t):={\rm Nr}_{F/\Q}g(t) \in \Z_S[t^\Z]$.
Then we have $(F(t))=(G(t^v))$ in $\wh{\Z}_S[[t^{\wh{\Z}}]]$. 
Now consider the inverse image of ideals by $\wh{\Z}[[t^{\wh{\Z}}]]\inj \wh{\Z}_S[[t^{\wh{\Z}}]]$. 
We may assume that the largest $p$-adic norms of coefficients of $f(t)$ and $g(t)$ are both 1 for every $p$ under $S$, so that we have the equality $(F(t))=(G(t^v))$ of ideals in $\wh{\Z}[[t^{\wh{\Z}}]]$. 
Now Lemma \ref{lem.rn} assures that $F(t)$ and $G(t)$ in $\Z[t^\Z]$ belong to the same Hillar class.  
If in addition $f(t)$ and $g(t)$ are reciprocal, then we have $F(t)\,\dot{=}\,G(t)$. 

For any $\alpha \in \ol{\Q}$ and any prime number $p$, we have $\mh(t-\alpha)=\mh(t^{-1}-\alpha)$ and $\mh_p(t-\alpha)=\mh_p(t^{-1}-\alpha)$. 
Hence the Hillar class of ${\rm Nr}_{F/\Q}\Delta_\rho(t)$ determines the Mahler measures 
$\mh({\rm Nr}_{F/\Q}\Delta_\rho(t))$ and $\mh_p({\rm Nr}_{F/\Q}\Delta_\rho(t))$. 
\end{proof} 

\section{Iwasawa modules} %and roots of $\Delta_\rho(t)$} \label{Sec.Iwasawa} 

In this section, we prove Theorem \ref{thm} (2) by considering the character decomposition of the Iwasawa module associated to a surjective homomorphism $\wh{\Z}\surj \Z/m\Z\times \Zp$ and using the Iwasawa isomorphism. 
Basic references are \cite{Washington} and \cite{Ueki4}. 

\begin{proof}[Proof of Theorem \ref{thm} (2)]

For each prime ideal $\p$ of the ring $O=O_{F,S}$, let $O_\p$ denote the $\p$-adic completion of $O$, 
namely, the inverse limit $\varprojlim_n O/\p^n$ endowed with a natural topology. 
The Chinese remainder theorem yields an isomorphism $\wh{O}\cong \prod_\p O_\p$. 
Note that $\wh{O}$ is not an integral domain, while so is $O_\p$. 
We write $\Cp[[\phantom{\bullet}]]:=O_\p[[\phantom{\bullet}]]\otimes_{O_\p}\Cp$. 
We consider the map 
$$\wh{O}[[t^\wh{\Z}]]\surj O_\p[[t^{\Z/m\Z\times \Zp}]]\inj \Cp[[t^{\Z/m\Z\times \Zp}]]\congto  \prod_{\xi^m=1} \Cp[[t^{\Zp}]]\congto \prod_{\xi^m=1}\Cp[[T]]$$
given by the composite of the natural surjective homomorphisms $\wh{O}\to O_\p$ of the coefficient rings and $t^\wh{\Z} \surj t^{\Z/m\Z\times \Zp}$ of groups, the tensor product $\otimes_{O_\p}\Cp$, the correspondence $f(t)\mapsto f(t\xi)$ for each $\xi$ with $\xi^m=1$, 
and the Iwasawa isomorphism $\Cp[[t^{\Zp}]]\cong \Cp[[T]];t\mapsto 1+T$ on each component. 
For each $t- \alpha \in \Cp[t^{\Z/m\Z \times \Zp}]\subset \Cp[[t^{\Z/m\Z \times \Zp}]]$, 
the map $\Cp[[t^{\Z/m\Z \times \Zp}]]\surj \Cp[[t^{\Zp}]]\congto \Cp[[T]]$ on the component of each $\xi$ yields the correspondence $(t-\alpha)\mapsto (t\xi-\alpha)=(t-\alpha/\xi)\mapsto (1+T-\alpha/\xi)=(T-(\alpha/\xi-1))=(T-(\alpha-\xi)/\xi)$ of ideals. 
By the $p$-adic Weierstrass preparation theorem \cite[Theorem 7.3]{Washington}, for $\beta \in \C_p$, 
we have $|\beta|_p<1$ if and only if $(T-\beta)$ is not a unit, and hence the value $\beta$ is determined by the ideal $(T-\beta)$. 
If $\alpha \in \C_p$ with $|\alpha|_p=1$ and $|\alpha-\zeta|_p<1$, then the image of $(t-\alpha)$ is nontrivial only at the component corresponding to $\xi=\zeta$, and determines the value $\alpha/\zeta-1$. 

Let $v\in \wh{\Z}^*$ and let $(u,w)$ denote the image of $v$ under the map $\wh{\Z}\surj \Z/m\Z\times \Zp$. 
By the same argument as above for $s=t^v\in t^{\wh{\Z}}$, the image of $t-\alpha$ is nontrivial only for $\xi=\zeta^u$, and determines the values $(\alpha/\zeta)^w$ and $\zeta^u$. 

Now let $(f(t))$ be an ideal of $\wh{O}[[t^{\wh{\Z}}]]$ generated by an unknown polynomial $f(t)\in O[t]$, 
assuming that the image of $t^\Z$ is not specified in $t^{\wh{\Z}}$. 
Take $p$ so that $f(t)$ is monic in $\Cp[t]$ and every nonzero root $\alpha_i \in \Cp$ satisfy $|\alpha_i|_p=1$. 
Suppose that every nonzero root satisfies $|\alpha_i^m-1|_p<1$ for $m\in \N_{>0}$. 
Then the family of $\alpha_i^m$'s with multiplicity is determined, again via the Iwasawa isomorphism and the $p$-adic Weierstrass preparation theorem. 

Recall the setting of Theorem \ref{thm} and the paraphrased Question \ref{question}, and 
note that $v\in \wh{\Z}^*$ is not specified. By taking $m$ as above, the families 
$\{[(\alpha_i/\zeta_i)^w] \mid w \in \Zp^*\}$ and $\{[\zeta_i^u]\mid u\in (\Z/m\Z)^*\}$ are determined. 
\end{proof} 

The Iwasawa module of $\Zp\times \Z/m\Z$-cover decomposes into direct sum by $\Z/m\Z$-characters.
The Fitting ideal of each component coincides with the ideal of $\C_p[[T]]$ above \cite[\S 13.4]{Washington}. 

\begin{rem} Let $k'$ denote the decomposition field of $f(t)$ and put $d'=[k':\Q]$. Then Proposition \ref{propA} assures that for any but finite number of prime number $p$ with $d'|(p-1)$, we have $m=p-1$, namely, every nonzero root $\alpha_i \in \Cp$ of $f(t)$ satisfies $|\alpha_i|_p=1$ and $|\alpha^m-1|_p<1$. 
\end{rem} 

\begin{rem} Recall that for each $(\alpha_i)_i \in \ol{\Q}^d$, we denote the equivalence class by permutation of indexes by $[\alpha_i]_i=[\alpha_1,\cdots,\alpha_d] \in \ol{\Q}^d/\mca{S}_d$. 
For each $A=[\alpha_i]_i$ and $n\in \Z$, we put $A^n=[\alpha_i^n]_i \in [\ol{\Q}^d]$. 

If $m,n\in \Z$ and $A,B\in [\ol{\Q}^d]$ with $A^m=B^m$ and $A^n=B^n$, 
we do \emph{not} necessarily have $A^g=B^g$ for $g={\rm gcm}(m,n)$. 
Indeed, let $\zeta=\zeta_{12}$ be a primitive 12-th root of unity and put $A=[\zeta,\zeta^8]$, $B=[\zeta^4,\zeta^5]$. 
Then we have $A^3=B^3=[1,\zeta^3]$, $A^4=B^4=[\zeta^4,\zeta^8]$. 
\end{rem} 

\section{Proofs of Propositions \ref{propA} and \ref{propB}}

We invoke some algebraic number theory to prove Propositions \ref{propA} and \ref{propB}. 
Although Proposition \ref{propA} is a well known fact, we give a proof for the convenience of the reader,
only assuming basics described in \cite{Morishita2012} and \cite{Ueki5}. 

\begin{lem} \label{lem-propA} 
(1) Let $p$ be a prime number and let $n\in \N_{>0}$. 
Let $\zeta_i$ be a primitive $i$-th root of unity for each $i\in \N_{>0}$. 
Then $\Qp(\zeta_{p^n-1})$ is the unique unramified extension of $\Qp$ of degree $n$, 
and is a cyclic extension with the Galois group generated by the Frobenius map $\zeta_{p^n-1}\mapsto \zeta_{p^n-1}^p$. 
Let $\F_{p^n}/\F_p$ denote the residue extension of $\Qp(\zeta_{p^n-1})/\Qp$. 
Then we have natural isomorphisms of Galois groups $\Gal(\Qp(\zeta_{p^n-1})/\Qp)\cong \Gal(\F_{q^n}/\F_q)\cong \Z/n\Z$. 
The Frobenius map corresponds to the multiplication by $p$ in $\Z/n\Z$. 
(Note that there would exist some $\nu< p^n-1$ with $\Qp(\zeta_{\nu})=\Qp(\zeta_{p^n-1})$.) 

(2) A prime number $p$ is ramified in a finite extension $k/\Q$ if and only if $p$ divides the discriminant $\mathfrak{d}_k$ of $k/\Q$. Hence only finite number of prime numbers $p$ are ramified in $k/\Q$. 
\end{lem} 

\begin{proof}[Proof of Proposition \ref{propA}] 
We assume that $p$ is unramified in $k/\Q$, noting that only finite number of prime numbers are ramified in $k/\Q$ by Lemma \ref{lem-propA} (2). 
Put $k_p=\Qp k$, and let $\F_{p^n}$ denote the residue field of $k_p$. 
Since $k_p/\Qp$ is an unramified finite extension, Lemma \ref{lem-propA} (1) assures that $\Gal(k_p/\Qp)\cong \Gal(\F_{p^n}/\F_p)\cong \Z/n\Z$. 

Put $l=k\cap \Qp$. Then $k/l$ is a Galois extension with $\Gal(k/l)\cong \Gal(k_p/\Qp)$. 
Indeed, let $\p$ denote the prime ideal over $(p)$ defined by $\ol{\Q}\inj \ol{\Q}_p$ and let $D<\Gal(k/\Q)$ denote the decomposition group of $\p$. The action of $D$ on $k/\Q$ induces a continuous action of $D$ on $k_p/\Q_p$ and hence a group homomorphism $\imath: D\inj \Gal(k_p/\Qp)$, which is an injection by $k\subset k_p$. 
Conversely, $\Gal(k_p/\Qp)$ acts on $k$ by restriction and induces a group homomorphism $j:\Gal(k_p/\Qp)\inj \Gal(k/\Q)$. Since $k$ is dense in $k_p$, this map $j$ also is an injection. 
By the uniqueness of extension of $p$-adic valuation, $p$-adic valuation of $k_p$ is stable by the action of ${\rm Im}(\jmath)$. Hence we have ${\rm Im}(\jmath)\subset D$, and $j$ decomposes as $j:\Gal(k_p/\Qp)\overset{\jmath}{\inj} D\inj\Gal(k/\Q)$. 
By definition, $D\overset{\imath}{\inj}\Gal(k_p/\Qp)\overset{\jmath}{\inj} D$ is an identity map. Since $j$ is an injection, $\imath$ and $\jmath$ are inverse maps to each other, and we have an isomorphism $\iota: D\congto \Gal(k_p/\Qp)$. 
Hence $k^D\subset \Qp$. Since $k^D\subset k$, we have $k^D\subset k\cap \Qp=l$. 
On the other hand, since $D={\rm Im}(\jmath)$, we have $l\subset k_p^D$. 
Therefore we have $k^D=l$ and $D=\Gal(k/l)$. 

Since $k/l$ is a subextension of $k/\Q$, we have $[k:l]=n$ and hence $n|d$. 
Recall that the Frobenius map corresponds to the multiplication by $p$ in $\Z/n\Z$.  
If $d|(p-1)$, then $n|(p-1)$, and the Frobenius map is trivial. 
Therefore we have $\F_q=\F_p$, $k_p=\Qp$, and $\alpha\in \Q_p$. 

Now let $\omega:\F_p\inj \Qp$ be a Teichm\"{u}ller lift. For all but finite number of $p$, we have $|\alpha|_p=1$. 
For such $p$ and $\alpha$, the image $\zeta$ of $\alpha$ via $\Qp\surj \F_p\overset{\omega}{\inj} \Q_p$ is a $(p-1)$-th root of unity. Hence we have $|\alpha-\zeta|_p<1$ and $|\alpha^{p-1}-1|_p<1$. 
\end{proof}

\begin{proof}[Proof of Proposition \ref{propB}]
Let $k/\Q$ be a finite Galois extension with $\alpha, \beta \in k$, 
let $k_p$ denote the closure of $k$ in $\ol{\Q}_p$, 
and let $\sigma \in \Gal(k_p/\Qp)$ with $\beta =\alpha^\sigma$. 
Then we have $\sigma^{[k_p:\Qp]}={\rm id}$ in $\Gal(k_p/\Qp)$. 
Since $\beta=\alpha^\sigma=\alpha^w$, we have 
$\alpha=\alpha^{\sigma^{[k_p:\Qp]}}=\alpha^{w^{[k_p:\Qp]}}$.
By the assumption that $\alpha$ is not a root of unity, we have $w^{[k_p:\Qp]}=1$. 
Hence $w$ is a root of unity. \end{proof}  

\section{Torsion growth}\label{sec.torsion} 

In this section, we refine results of \cite{Ueki4} and \cite{TangeRyoto2018JKTR}; 
Fox's formula and torsion growth formulas with use of Mahler measures of the twisted homology groups in a $\Z$-cover of a 3-manifold. 
%$\rho$-twisted analogues of Fox's formula and the asymptotic formulas of torsion growth in a $\Z$-cover with use of Mahler measures, 
%following \cite{Ueki4} and \cite{TangeRyoto2018JKTR}. 
We recall definitions and known results in \S11.1, state Theorem \ref{thm.mh} with some remarks in \S11.2, and prove them in \S11.3. 

\subsection{Backgrounds} \label{subsec.mh} 
Let $\Delta_K(t)$ denote the Alexander polynomial of a knot $K$ in $S^3$ and let $X_n\to X=S^3-K$ denote the $\Z/n\Z$-cover. 
Fox's formula (cf. \cite{Weber1979}) asserts that if $\Delta_K(t)$ does not vanish at roots of unity, 
then the order of the $\Z$-torsion subgroup satisfies $\ds |{\rm tor} H_1(X_n)|=\prod_{\zeta^n=1} |\Delta_K(\zeta)|=|{\rm Res}(\Delta_K(t),t^n-1)|$. 
In addition, we easily see that if we put $\Psi_n(t)={\rm gcd}(\Delta_K(t), t^n-1)$, then we have $|{\rm tor} H_1(X_n)|=c_n|{\rm Res}(\Delta_K(t), (t^n-1)/\Psi_n(t))|$ for a non-zero bounded sequence $c_n=|{\rm tor}H_1(X_\infty)/(\Psi_n(t))|$. 

We define the Mahler measure $\mh(f(t))$ of a polynomial $0\neq f(t)\in \C[t]$ by $\ds \log \mh(f(t))=\int_{|z|=1}\frac{\log|f(z)|}{z}dz$, which is naturally interpreted even if $f(t)$ has roots on $|z|=1$. 
If $f(t)=a_0\prod_i (t-\alpha_i)$ in $\C[t]$, then Jensen's formula asserts $\ds \mh(f(t))=|a_0|\prod_{|\alpha_i|>1}|\alpha_i|$.  
Since the integral over $|z|=1$ coincides with the limit of the mean value of $f(\zeta)$ for $\zeta^n=1$, 
Fox's formula yields the asymptotic formula $|{\rm tor}H_1(X_n)|^{1/n}=\mh(\Delta_K(t))$ of torsion growth \cite{GS1991}. 

The author proved in \cite{Ueki4} that if we replace the absolute value by the $p$-adic norm with $|p|=p^{-1}$, 
then a similar argument yields $|{\rm tor}H_1(X_n)|_p^{1/n}=\mh_p(\Delta_K(t))$. 
Here we define our $p$-adic Mahler measure 
$\mh_p(f(t))$ of  $f(t)\in \Cp[t]$ by $\ds \log \mh_p(f(t))=\lim_{n\to \infty}\frac{1}{n}\sum_{\zeta^n=1}\log|f(\zeta)|_p$. 
If $f(t)=a_0\prod_i (t-\alpha_i)$ in $\C_p[t]$, then Jensen's formula asserts $\ds \mh_p(f(t))=|a_0|_p \prod_{|\alpha_i|_p>1} |\alpha_i|_p$. 
(Note that this $\mh_p$ differs from that introduced by Besser--Deninger with use of Iwasawa's $p$-adic logarithm in \cite{BesserDeninger1999}.) 

Now let $\rho:\pi_K\to \GL_N(O)$ be a knot group representation over the ring $O=O_{F,S}$ of $S$-integers of a number field $F$. 
Analogues of those results above for the twisted Alexander polynomial $\Delta_\rho(t)$ and the twisted homology groups $H_1(X_n,\rho)$ 
were established by Tange \cite{TangeRyoto2018JKTR}. 
Suppose that $O$ is a UFD so that $\Delta_{\rho,i}(t)$ are defined for $i=0,1$. 
If $\rho$ is irreducible, namely, all residual representations are irreducible, then we have %$H_0(X_\infty, \rho)=0$ 
$\Delta_{\rho,0}(t)\,\dot{=}\,1$ \cite[Corollary 3]{TangeRyoto2018JKTR}. 
Suppose $\Delta_{\rho,0}(t)\,\dot{=}\,1$, so that we have $\Delta_\rho(t):=\Delta_{\rho,1}(t)\,\dot{=}\,W_\rho(t)$. 
Recall $\ol{\Delta}_{\rho}(t)={\rm Nr}_{F/\Q}\Delta_\rho(t)$. 
The Wang exact sequence induces a natural isomorphism $H_1(X_\infty,\rho)/(t^n-1)H_1(X_\infty,\rho)\cong H_1(X_n,\rho)$. Therefore if $\Delta_{\rho}(t)$ does not vanish at roots of unity, then we have $\ds |H_1(X_n,\rho)|\underset{S}{=}|\prod_{\zeta^n=1}\ol{\Delta}_{\rho}(\zeta)|=|{\rm Res}(\ol{\Delta}_{\rho}(t),t^n-1)|<\infty$. 
Here $\underset{S}{=}$ indicates equality up to multiplication by $S$. 
If $S=\emptyset$, then we have $\ds \lim_{n\to \infty} |H_1(X_n,\rho)|^{1/n}=\mh(\ol{\Delta}_\rho(t))$. 

\subsection{Theorem} 
Our aim here is to remove several assumptions above.  
Let $O=O_{F,S}$ denote the ring of $S$-integers of a number field $F$, which is not necessarily a UFD. 
Let $\pi$ be a finite type discrete group with a surjective homomorphism $\alpha:\pi\surj t^\Z$ and a representation $\rho:\pi\to \GL_N(O)$. 
Let $X$ be a connected manifold with $\pi_1(X)=\pi$ and let $X_\infty\to X$ and $X_n\to X$ denote the $\Z$-cover and $\Z/n\Z$ cover corresponding to $\alpha$ and $(\Z\surj \Z/n\Z)\circ \alpha$ for each $n\in \N_{>0}$. 
We define \emph{the norm polynomial} $\ol{\Delta}_\rho(t) \in \Z[t^\Z]$ of $\rho$ to be a generator of ${\rm Nr}_{F/\Q}\wt{{\rm Fitt}}H_1(X_\infty,\rho) \cap \Z[t^\Z]$. 
If $\Delta_\rho(t) \in O[t^\Z]$ is defined, then we have $\ol{\Delta}_\rho(t)$ $\dot{=}$ ${\rm Nr}_{F/\Q}\Delta_\rho(t)$ in $\Z_S[t^\Z]$. 

\begin{thm} \label{thm.mh} Suppose that $\rho:\pi\to \GL_N(O)$ is non-abelian and absolutely irreducible. Then, 

(1) For each $n\in \N_{>0}$, the Wang exact sequence induces a natural isomorphism 
$$H_1(X_\infty,\rho)/(t^n-1)H_1(X_\infty,\rho) \congto H_1(X_n,\rho).$$

(2) Put $\Psi_n(t)={\gcd}(\ol{\Delta}_\rho(t),t^n-1)$ and $r_n={\rm Res}(\ol{\Delta}_\rho(t), (t^n-1)/\Psi_n(t))$. 
Then the equality $$\ds |{\rm tor}H_1(X_n,\rho)|=c_n|r_n|\prod_{p\in S} |r_n|_p$$ 
holds, where $c_n$ is a bounded sequence defined by $c_n=|{\rm tor} H_1(X_\infty,\rho)/\Psi_n(t)|$. 

(3) For any prime number $p$, the asymptotic formulas  
$$\ds\lim_{n\to \infty} |{\rm tor}H_1(X_n,\rho)|^{1/n}=\mh (\ol{\Delta}_\rho(t)) \ \ {\rm and}\ \ 
\ds\lim_{n\to \infty} ||{\rm tor}H_1(X_n,\rho)||_p^{1/n}=\mh_p (\ol{\Delta}_\rho(t))$$
of torsion growth hold. 
\end{thm} 

A representation $\rho$ is said to be \emph{non-abelian} if it does not factors through the abelianization map $\pi\surj \pi^{\rm ab}$. 
If $\rho$ is non-abelian, then $\Delta_{\rho,i}(t)\neq 0$ (cf. \cite{FriedlKimKitayama2012}). 
A representation $\rho$ is said to be \emph{absolutely irreducible} if $\rho\otimes_O \ol{\Q}$ is an irreducible representation. This condition is much weaker than being irreducible in the sense of \cite{TangeRyoto2018JKTR}, namely, every residual representation being irreducible. Nevertheless, we have important irreducible representations over $O$ such as the Holonomy representations of hyperbolic knots. 
An (absolutely) irreducible representation $\rho:\pi\to \GL_N(O)$ with $N\geq 2$ is known to be non-abelian. 

Even if we do not assume that $\rho$ is absolutely irreducible, we still have an injective homomorphisms $p_n:H_1(X_\infty, \rho)/(t^n-1)H_1(X_\infty, \rho) \inj H_1(X_n,\rho)$ and $\Coker(p_n) \inj H_0(X_\infty,\rho)$. 

If the Dedekind domain $O$ is a UFD (hence a PID), then ${\rm Res}(\ol{\Delta}_\rho(t),t^n-1)$ may be replaced by ${\rm Nr}_{F/\Q}{\rm Res}(\Delta_\rho(t),t^n-1)$. 

The reasons why we consider $O=O_{F,S}$ instead of $O_F$ have been that 
(i) the image of any $\rho:\pi \to \GL_N(\C)$ may be contained in some $O_{F,S}$ up to conjugate,  
and (ii) we may take some $S$ so that $O_{F,S}$ is PID and hence ${\rm \Delta}_\rho(t)$ is defined. 
Theorem \ref{thm.mh} removes the reason (ii). 

\subsection{Proof}  
Here we regard several lemmas of general algebra described in \cite{TangeRyoto2018JKTR} and some $p$-adic number theory as basic facts and use them rather freely. 

\begin{proof}[Proof of Theorem \ref{thm.mh} (1)]
The Wang short exact sequence of twisted complex induces the following long exact sequence 
\[\cdots \to H_1(X_\infty,\rho)\underset{t^n-1}{\to} H_1(X_\infty,\rho)\underset{p_n}{\to} H_1(X_n,\rho)\underset{\del}{\to} H_0(X_\infty,\rho)\underset{t^n-1}{\to} H_0(X_\infty,\rho) \to \cdots .\]

Let $\p\neq 0$ be any prime ideal of $O$, 
and let $F_{(\p)}$ and $O_{(\p)}$ denote the localizations of $F$ and $O_F$ at $\p$. 
Since $\rho$ is absolutely irreducible, $\rho$ is irreducible over $F_{(\p)}$.  
By \cite[Proposition A3]{FriedlKimKitayama2012}, the Fitting ideal of $H_0(X_\infty,\rho)\otimes F_{(\p)}$ over $F_{(\p)}[t^\Z]$ is $(1)$, 
and hence the ideal $\wt{\rm Fitt} H_0(X_\infty,\rho)\otimes O_{(\p)}$ is generated by a number in $O_{(\p)}$. 
This implies that the map $t^n-1$ on $H_0(X_\infty,\rho)\otimes O_{(\p)}$ is injective, and so is $t^n-1$ on $H_0(X_\infty,\rho)$. 
Therefore we have  $\Coker p_n\cong \Im \del=0$ and an isomorphism $H_1(X_\infty,\rho)/(t^n-1)H_1(X_\infty,\rho)\congto H_1(X_n,\rho)$. 
\end{proof} 

\begin{proof}[Proof of Theorem \ref{thm.mh} (2)] 
Assume $\Psi_n(t)=1$. 
Put $\mca{H}=H_1(X_\infty,\rho)$ and $\mf{r}_n=\{{\rm Res}(f(t), t^n-1)\mid f(t)\in \wt{\rm Fitt}\mca{H}\}\subset O$. 
For the product of several distinct prime ideals $0\neq \P=\prod_j \p_j \subset O$, let $O_{(\P)}$ denote the localization at $\P$. 
For each ideal $\mf{a}\subset O$, we write $\mf{a}_{(\P)}=\mf{a}O_{(\P)}$. 
By Hillman's Theorem \cite[Theorem 3.13]{Hillman2}, we have 
$|({\rm tor}\mca{H}/(t^n-1)\mca{H})\otimes O_{(\p)}|_p=|O_{(\p)}/\mf{r}_{n\,(\p)}|_p$ for each $\p|(p)$,  
hence 
$|{\rm tor}\mca{H}/(t^n-1)\mca{H}|_p
=\prod_{\p|(p)}|({\rm tor}\mca{H}/(t^n-1)\mca{H})\otimes O_{(\p)}|_p
=\prod_{\p|(p)}|O_{(\p)}/\mf{r}_{n\,(\p)}|_p
=|O_{(p)}/\mf{r}_{n\,(p)}|_p
=|\Z_{S\,(p)}/{\rm Nr}_{F/\Q}\mf{r}_{n\,(p)}|_p
=|\Z_S/{\rm Nr}_{F/\Q}\mf{r}_n|_p
=|\Z_S/{\rm Res}(\ol{\Delta}_\rho(t),t^n-1)|_p
=|{\rm Res}(\ol{\Delta}_\rho(t),t^n-1)|_p$. 
Since an $O$-module has no $p$-torsion for $p\in S$, 
we obtain $|{\rm tor}H_1(X_n,\rho)|=|{\rm tor}\mca{H}/(t^n-1)\mca{H}|=\prod_{p\not \in S} |{\rm tor}\mca{H}/(t^n-1)\mca{H}|_p =\prod_{p\not \in S} |{\rm Res}(\ol{\Delta}_\rho(t),t^n-1)|_p.$ 

Next, suppose that $\Psi_n(t)$ is not necessarily 1, and put $C_n:=|{\rm tor} \mca{H}/\Psi_n(t)|$. Then $(C_n)_n$ is a bounded sequence, since $\{\Psi_n(t) \mid n\in \N\}$ is a finite set. Replacing $t^n-1$ by $(t^n-1)/\Psi_n(t)$ in the argument above, we obtain the desired formula. 
\end{proof}

\begin{proof}[Proof of Theorem \ref{thm.mh} (3)] 
Note that finite number of roots of unity may be skipped in the definition of ($p$-adic) Mahler measure.  
By the $p$-adic asymptotic formula \cite[Theorem 2.5]{Ueki4} and the assumption, we have $\ds \lim_{n\to \infty}|r_n|_p^{1/n}=\mh_p(\ol{\Delta}_\rho(t))=1$ for $p \in S$. 
Since $\ds \lim_{n\to \infty}C_n^{1/n}=1$, by taking $\ds \lim_{n\to \infty} \bullet^{1/n}$ in the equality  
$|{\rm tor}H_1(X_n,\rho)|=|r_n|\prod_{p\in S} |r_n|_p\times C_n$ in (2), we obtain the desired formula 
$\ds \lim_{n\to \infty}|{\rm tor}H_1(X_n,\rho)|^{1/n}=\ds \mh(\ol{\Delta}_\rho(t))$ and $\ds \lim_{n\to \infty}||{\rm tor}H_1(X_n,\rho)||_p^{1/n}=\ds \mh_p(\ol{\Delta}_\rho(t))$
\end{proof}  

\section{Riley's representations} \label{Sec.Examples} 
We revisit examples of twisted Alexander polynomials of Riley's parabolic representations of 2-bridge knot groups  exhibited in \cite[Section 9]{TangeRyoto2018JKTR} and \cite[Example 2.3]{HirasawaMurasugi2010}. 

\begin{eg} \label{eg.Riley} 
The knot group $\pi$ of a two-bridge knot $K$ admits a standard presentation $\pi=\langle a,b \mid aw=wb\rangle$. 
Let $u=\alpha$ be a root of Riley's polynomial $\Phi_K(1,u)\in \Z[u]$ and let $O$ denote the ring of integers of $F=\Q(\alpha)$. Then $u=\alpha$ corresponds to a parabolic irreducible representation $\rho:\pi\to \SL_2(O)$ defined by  $\rho(a)=\spmx{1&1\\0&1}$ and $\rho(b)=\spmx{1&0\\-u&1}$. If $O$ is a UFD (hence PID), then $\Delta_\rho(t):=\Delta_{\rho,1}(t) \in O[t^\Z]$ is defined and satisfies $\ol{\Delta}_\rho(t)={\rm Nr}_{F/\Q}\Delta_\rho(t)$. 
In addition, since any residual representation of this $\rho$ is irreducible, \cite[Corollary 3.1]{TangeRyoto2018JKTR} yields that $\Delta_{\rho,0}(t)\,\dot{=}\,1$ and $W_\rho(t)\,\dot{=}\,\Delta_{\rho}(t)$. 
If $K$ is hyperbolic, then (a lift of) the holonomy representation is given in this way. (cf. \cite{Riley1972PLMS, Riley1984}, \cite[Lemma 5]{DuboisHuynhYamaguchi2009}, \cite[Example 2.3]{HirasawaMurasugi2010}, \cite[Section 9]{TangeRyoto2018JKTR}.) 
Let $p$ be any prime number. 

(i) Let $K$ be the trefoil $3_1$. Then $\Delta_\rho(t)=t^2+1=\Phi_4(t) \in \Z[t]$. 
We have $\ol{\Delta}_\rho(t)=\Phi_4(t)^2$, $\ds \lim_{n\to \infty}|{\rm tor}H_1(X_n,\rho)|^{1/n}=\mh(\ol{\Delta}_\rho(t))=1$ and $\ds \lim_{n\to \infty}|{\rm tor}H_1(X_n,\rho)|_p^{1/n}=\mh_p(\ol{\Delta}(t))=1$. 

(ii) Let $K$ be the figure-eight knot $4_1$. Then $\Delta_\rho(t)=t^2-4t+1 \in \Z[\frac{1+\sqrt{-3}}{2}]$, 
where $O=\Z[\frac{1+\sqrt{-3}}{2}]$ is a UFD. 
We have $\ol{\Delta}_\rho(t)=\Delta(t)^2$, $\ds \lim_{n\to \infty}|{\rm tor}H_1(X_n,\rho)|^{1/n}=\mh(\ol{\Delta}_\rho(t))=(2+\sqrt{3})^2=7+4\sqrt{3}$, $\ds \lim_{n\to \infty}|{\rm tor}H_1(X_n,\rho)|_p^{1/n}=\mh_p(\ol{\Delta}_\rho(t))=1$.  

(iii) Let $K=5_1$. Then $\Delta(t)=(t^2+1)(t^4-\frac{1+\sqrt{5}}{2}t^2+1)=\Phi_4(t)$$\phi_{20}^+(t) \in \Z[\frac{1+\sqrt{5}}{2}]$, 
where $O=\Z[\frac{1+\sqrt{5}}{2}]$ is a UFD. 
We have $\ol{\Delta}_\rho(t)=(t^2+1)^2(t^8-t^6+t^4-t^2+1)=\Phi_4(t)^2\Phi_{20}(t)$, hence $\ds \lim_{n\to \infty}|{\rm tor}H_1(X_n,\rho)|^{1/n}=\mh(\ol{\Delta}_\rho(t))=1$ and $\ds \lim_{n\to \infty}|{\rm tor}H_1(X_n,\rho)|_p^{1/n}=\mh_p(\ol{\Delta}_\rho(t))=1$. 

(iv) Let $K=5_2$, which is the 2-bridge knot $K(3/7)$ of type $(3,7)$ in \cite[Example 2.3 (3)]{HirasawaMurasugi2010}, so that we have $\pi=\langle x,y\mid wx=yw \rangle$ with $w=xyx^{-1}y^{-1}xy$, $\Phi_K(1,u)=u^3+u^2+2u+1$, and Wada's invariant $W_\rho(t)=(4+\alpha^2)t^2-4t+(4+\alpha^2)$ for $\rho\otimes \C$ corresponding to a root $u=\alpha$ of $\Phi_K(1,u)$. 
PARI/GP \cite{PARI2} tells that the class number of $F=\Q(\alpha)$ is 1, hence $O_F=\Z[\alpha]$ is a PID. 
(In addition, the discriminant is $-23$, hence only $p=23$ is ramified in $F/\Q$.)  
Thus $\Delta_\rho(t)$ is defined and coincides with $W_\rho(t)$. 
(Even if $O$ is not a UFD, we still have $\ol{\Delta}_\rho(t)={\rm Nr}_{F/\Q}W_\rho(t)$.) 
We have $\ol{\Delta}_\rho(t)={\rm Nr}_{F/\Q}\Delta_\rho(t)=\prod_{\Phi(1,u)=0}((4t^2-4t+4)+u^2(t^2+1))
=(4t^2-4t+4)^3-3(4t^2-4t+4)^2(t^2+1)+2(4t^2-4t+4)(t^2+1)^2+(t^2+1)^3 
%=-89 t^6 + 232 t^5 - 475 t^4 + 528 t^3 - 475 t^2 + 232 t - 89.$ 
%=-25 t^6 + 104 t^5 - 219 t^4 + 272 t^3 - 219 t^2 + 104 t - 25
=25 t^6 - 104 t^5 + 219 t^4 - 272 t^3 + 219 t^2 - 104 t + 25$. 
By Theorem \ref{thm.mh}, we have $\ds \lim_{n\to \infty}|{\rm tor}H_1(X_n,\rho)|^{1/n}=\mh(\ol{\Delta}_\rho(t))$ and $\ds \lim_{n\to \infty}|{\rm tor}H_1(X_n,\rho)|_p^{1/n}=\mh_p(\ol{\Delta}_\rho(t))=1$. 
\end{eg} 

Theorem \ref{thm} (1) assures that 
the isomorphism class of $\wh{\rho}$ determines $\ol{\Delta}_\rho(t)$ of (i)--(iv), hence $\Delta_\rho(t)$ of (i) and (ii). Indeed, $\Delta_\rho(t)$ in (i) and (ii) are recovered from the image by ${\rm Nr}_{F/\Q}$, only by noting that they should belong to $O[t^\Z]$. 
In (i) for instance, even if we take a larger field $F$ so that $t\pm \sqrt{-1} \in O[t^\Z]$, 
we may still say $\Delta_\rho(t)\neq (t\pm \sqrt{-1})^2$ hence $\Delta_\rho(t)=t^2+1$ by Theorem \ref{thm} (2). 
In (iii), we cannot distinguish divisors $\phi_{20}^{\pm}=t^4-\frac{1\pm\sqrt{5}}{2}t^2+1 \in \Z[\frac{1+\sqrt{5}}{2}]$ of $\Phi_{20}(t)$. 
This ambiguity is inevitable, since these two factors correspond to each other by $t\mapsto t^v$ for some $v\in \wh{\Z}^*$. 
If there is some automorphism $\tau$ of $\pi$ such that $\Delta_{\rho\circ \tau}(t)=
(t^2+1)(t^4-\frac{1-\sqrt{5}}{2}t^2+1)=\Phi_4(t)\phi_{20}^-(t)$ holds, then this ambiguity is not essential. 
In (iv), $\ol{\Delta}_\rho(t)$ gives three candidates for $\Delta_\rho(t) \in \Z[\alpha][t^\Z]$.
Since the holonomy representation corresponds to both of non-real roots of $u^3+u^2+2u+1$, we may eliminate one candidate and determine $\Delta_\rho(t)$ up to complex conjugates.\\  

Here we attach remarks on the topological entropy $h$ and its \emph{purely} $p$-adic analogue $\hbar_p$ of the solenoidal dynamical system defined as the dual of the meridian action on the twisted Alexander module. 

The polynomial $\ol{\Delta}_\rho(t)$ in Example \ref{eg.Riley} (iv) coincides with $\Delta_{\rho_\Phi}(t)$ in \cite[Example 4.5]{SilverWilliams2009TA} 
associated to the total representation $\rho_\Phi:\pi\to \GL_8(\Z)$ of the Riley polynomial $\Phi_K(1,u)=u^3+u^2+2u+1$. 
The topological entropy is given by the Mahler measure as $h=\log \mh(\ol{\Delta}_\rho(t))$.  

If $\ol{\Delta}_\rho(t)$ does not vanish on $|z|_p=1$ (this may happen only if it is not monic), then Besser--Deninger's \emph{purely} $p$-adic log Mahler measure $m_p$ is defined with use of Iwasawa's $p$-adic logarithm, which is different from our $\log \mh_p$ but still satisfies the Jensen formula (\cite{BesserDeninger1999}, see also \cite{Ueki4}). By \cite[Theorem 1.4]{Katagiri2019arXiv}, which extends \cite{Deninger2009AAG}, the \emph{purely} $p$-adic periodic entropy is given by $\hbar_p=m_p({\rm Nr}_{F/\Q}\Delta_\rho(t))$. 

\section{Hyperbolic volumes} \label{sec.vol} 

An optimistic conjecture of L\"uck implies the profinite rigidity of hyperbolic volumes (cf. \cite[Section 7]{Reid-ICM2018}). 
One may ask if the hyperbolic volume ${\rm Vol}(K)$ of a hyperbolic knot $K$ is determined by the isomorphism class of the profinite completion $\wh{\rho}_{\rm hol}$ of the holonomy representation $\rho_{\rm hol}:\pi_K\to \SL_2(O)\subset \SL_2(\C)$.
In this section, we prove Corollary \ref{cor.vol} to partially answer this question. 

Recall that Wada's initial definition of $W_\rho(t)$ in \cite{Wada1994} may be regarded with less indeterminacy, so that it is defined up to multiplication by $\pm t^{\pm1}$, and coincides with the Reidemeister torsion $\tau_\rho(t)=\tau_{\rho\otimes \alpha}(S^3-K) \in F[t^\Z]$. 

Let $\rho_n$ denote the composite of $\rho_{\rm hol}$ and the symmetric representation $\SL_2(O)\to \SL_n(O)$ for each $n>1$. 
Based on the theory of analytic torsions due to M\"uller, Menal-Ferrer, Porti \cite{Muller1993JAMS, Muller2012B, MenalFerrerPorti2014} and the relationship among several torsions proved Kitano and Yamaguchi \cite{Kitano1996PJM, Yamaguchi2008Fourier}, Goda %\cite{Goda2017pja} 
proved the following formula. (Recently it was announced that 
B\'enard--Dubois--Heusener--Porti \cite{BenardDuboisHeusenerPorti2019arXiv} replaced $t=1$ 
%proved its generalization where $t=1$ is replaced 
by $t=\zeta$ for any $\zeta\in \C$ with $|\zeta|=1$.) %root of unity $\zeta$.) 

\begin{prop}[{\cite[Theorem 1.1]{Goda2017pja}}]
For each $m>0$, put $A_{2m}(t)=\tau_{\rho_{2m}}(t)/\tau_{\rho_2}(t)$ and $A_{2m+1}(t)=\tau_{\rho_{2m+1}}(t)/\tau_{\rho_3}(t)$. 
Then $\ds \lim_{m\to \infty} \frac{\log |A_{2m+1}(1)|}{(2m+1)^2}=\lim_{m\to \infty} \frac{\log |A_{2m}(1)|}{(2m)^2}=\frac{{\rm Vol}(K)}{4\pi}$ holds.
\end{prop}

By Remark \ref{rem.WDH}, if the modulus $|u|$ of any $u \in O$ is determined by $|{\rm Nr}_{F/\Q} u|$, then $\wh{\rho}$ determines the hyperbolic volume. 
The knots in Examples \ref{eg.Riley} (ii) and (iv) are hyperbolic, while (i) and (iii) are not.  
In the case of the figure eight knot (ii), since $u\in O$ satisfies ${\rm Nr}_{F/\Q}u=|u|^2$, 
the volume ${\rm Vol}(K)$ is determined by $\wh{\rho}$. 
(Due to Boileau--Friedl \cite{BoileauFriedl2015} and Bridson--Reid \cite{BridsonReid2015}, 
the figure eight knot has the profinite rigidity amongst 3-manifold groups. See also \cite{BridsonReidWilton2017}.) 

In general, if the numbers $r_1$ and $2r_2$ of real and complex infinite places of a number field $F$ satisfy $r_1+r_2>1$, 
then $O=O_F$ contains a unit $u$ with the modulus $|u|\neq 1$. 
The field $F$ in (iv) is a complex cubic field and satisfies $r_1+r_2>1$. 
Hence the modulus $|\Delta_\rho(1)|$ cannot be recovered from the norm ${\rm Nr}_{F/\Q}\Delta_\rho(1)$. 
Note in addition that $O=\Z$ does not occur, since a subgroup of $\SL_2(\Z)$ is a Fuchsian group, while a hyperbolic knot group is a Kleinian group which is not Fuchsian (cf.~\cite{MaclachlanReid2003GTM}).  
At this moment by Theorem \ref{thm} (1) and Remark \ref{rem.WDH}, we may conclude Corollary \ref{cor.vol}: 
If %$O=\Z$ or 
$O=O_F$ for an imaginary quadratic field $F$, then the hyperbolic volume ${\rm Vol}(K)$ of $K$ is determined by the isomorphism class of $\wh{\rho}_{\rm hol}$. 
Further approaches are attached in Question \ref{furtherq}. 

\section*{Acknowledgments} 
I would like to express my gratitude to L\'eo B\'enard, Michel Boileau, Frank Calegari, Yuichi Hirano, Teruhisa Kadokami, Takenori Kataoka, Tomoki Mihara, Yasushi Mizusawa, Alan Reid, Ryoto Tange, Anastasiia Tsvietkova, Yoshikazu Yamaguchi, people I met at CIRM in Luminy and at GAU in G\"ottingen, and anonymous referees of the previous article and this article for fruitful conversations. 
This work was partially supported by JSPS KAKENHI Grant Number JP19K14538. 

%\small 
\bibliographystyle{amsalpha}%{amsplain}%{amsalpha}
%\bibliography{/Users/uekijun/Dropbox/refs1} 
\providecommand{\bysame}{\leavevmode\hbox to3em{\hrulefill}\thinspace}
\providecommand{\MR}{\relax\ifhmode\unskip\space\fi MR }
\providecommand{\MRhref}[2]{%
  \href{http://www.ams.org/mathscinet-getitem?mr=#1}{#2}
}
\providecommand{\href}[2]{#2}

%\noindent 

%Jun Ueki \\
%Department of Mathematics, School of System Design and Technology, Tokyo Denki University,\\
%5 Senju Asahi-cho, Adachi-ku, 120-8551, Tokyo, Japan\\  
%E-mail: \url{uekijun46@gmail.com}  

\newpage 

\setcounter{section}{-1}
\section{Erratum to ``Profinite rigidity for twisted Alexander polynomials''} 
%{\footnotesize  (J. reine angew. Math. 771 (2021), 171--192)}} 

%\begin{abstract} 
%{\footnotesize 
We clarify the definition of the divisorial hull and recollect some basic facts. Then we correct Lemma 4.2 and Theorem 11.2 (1), (2) in the original article. %} 
%\end{abstract}

\section*{Divisorial hull} 
\emph{The divisorial hull} $\wt{\mf{a}}$ of an ideal $\mf{a}$ of an integral domain $R$ is defined as 
the intersection of all principal \emph{fractional} ideals of $R$ containing $\mf{a}$, 
so that $\wt{\mf{a}}$ is an ideal of $R$. 
Let $R_0$ denote the field of fractions of $R$. 
For each $R$-submodule $\mf{b}$ of $R_0$, put 
$[R:\mf{b}]=\{x\in R_0\mid x\mf{b}\subset R\}$. 
Then Bourbaki \cite[$\S$1, Proposotion 1]{MR0260715} asserts that $\wt{\mf{a}}=[R:[R:\mf{a}]]$. 

\begin{lem} \label{lem} 
Let $R$ be an integrally closed Noetherian domain and let $\p$ run through the set of all height 1 prime ideals of $R$, so that the localization $R_\p$ at each $\p$ is a DVR and $R=\cap_\p R_\p$ holds in $R_0$. 

{\rm (1)} If $\mf{a}$ is an ideal of $R$, then $\wt{\mf{a}}=\cap_\p \mf{a}_\p$ holds. 

{\rm (2)} If $\mf{a}$ is an ideal of $R$ and $S$ a multiplicative set, then $\wt{\mf{a}_S}=\wt{\mf{a}}_S$ in the localization $R_S$. 

{\rm (3)} If ideals $\mf{a}, \mf{b}, \mf{c}$ of $R$ satisfy $\mf{a}_\p=\mf{b}_\p\mf{c}_\p$ for every $\p$, 
then $\wt{\mf{a}}=\wt{\mf{b}}\wt{\mf{c}}$ holds. 
\end{lem} 

%\begin{comment} 

\begin{proof} (Due to J.~Hillman.) 
(1) 
Since $R$ is Noetherian, we may easily verify for each $\p$ that $[R:\mf{a}]_\p=[R_\p:\mf{a}_\p]$, and hence $\wt{\mf{a}}_\p=\wt{\mf{a}_\p}$. 
Therefore we have $\wt{\mf{a}}\subset \cap_\p \wt{\mf{a}}_\p = \cap_\p \wt{\mf{a}_\p} =\cap_\p \mf{a}_\p$. 
On the other hand, if $\mf{a}\subset (a)$ with $0\neq a\in R$, then 
$\cap_\p \mf{a}_\p\subset \cap_\p(a)_\p =(a)(\cap_\p R_\p)$ $=(a)$. 
Hence $\cap_\p \mf{a}_\p \subset \wt{\mf{a}}$. 
Thus we have $\wt{\mf{a}}=\cap_{\p} \mf{a}_\p$ (cf.\cite[Lemma 3.2]{Hillman2}).

(2) It suffices to prove the assertion for $\mf{a}\neq 0$. 
Since $R$ is a Noetherian ring, there are at most finitely many minimal prime ideals containing $\mf{a}$ 
(cf.~\cite[Proposition 4.6, Theorem 7.13]{AtiyahMacdonald}). 
Since $\mf{a}$ is a nonzero ideal of a Noetherian domain, $\mf{a}$ is contained in at most finitely many height 1 primes. 
Thus, we have $\mf{a}_\p=R_\p$ for all but finite $\p$'s. Since $R=\cap_\p R_\p$, 
the expression $\wt{\mf{a}}=\cap_{\p} \mf{a}_\p$ may be regarded as a finite intersection. 
Since localization by $S$ commutes with finite intersections, we obtain the assertion. 

(3) 
Since $R$ is a Noetherian domain, 
the assumption implies that $\cap_{\p} \mf{a}_\p=\cap_\p \mf{b}_\p \cap_\p \mf{c}_\p$. Hence by (1), we obtain the assertion. 
\end{proof} 

\section*{Erratum 1}
Lemma 4.2 in \cite{Ueki6} generalizing \cite[Theorem 3.12 (3)]{Hillman2} ought to be stated as follows, where $R$ is assumed to be \emph{integrally closed} and the assumption on $K$ is removed. We attach a refined proof to clarify the argument. 
\setcounter{section}{4}
\setcounter{thm}{1}
\begin{lem}[corrected] \label{lem.dhFitt}
Let $0\to K \to N \to C \to 0$ be an exact sequence of finitely generated modules over an integrally closed Noetherian domain $R$, and suppose $r={\rm rank}(C)$ and $s={\rm rank}(K)$. Then we have $\wt{\rm Fitt}_{r+s}(M)=\wt{\rm Fitt}_s(K)\wt{\rm Fitt}_r(C)$. 
\end{lem}

\begin{proof} Let $\p$ run through all hight 1 prime ideals of $R$. 
A finitely generated $R$-module $M$ is of rank $r$ if $M\otimes_R R_0$ is of dimension $r$ as an $R_0$-vector space. 
If an $R$-module $M$ is of rank $r$, then the localization $M_\p$ at each $\p$ is of rank $r$ as an $R_\p$-module. 
Note that every $R_\p$ is a DVR and that any finite module over a DVR is of projective dimension $\leq 1$.
Hence we have presentation matrices $P(K_\p) \in M_{k+s,k}(R_\p)$, $P(C_\p)\in M_{c+r,c}(R_\p)$,  $P(M_\p)=\smat{P(K_\p)& 0\\ \ast& P(C_\p)}$ of ranks $k$, $c$, and $k+c$ for $K_\p$, $C_\p$, and $M_\p$ respectively.
Then the only nonzero elements of ${\rm Fitt}_{r+s}(M_\p)$ are those obtained by deleting $r$-rows and $s$-rows form $P(K_\p)$ and $P(C_\p)$ respectively and talking product of the resulting elements of ${\rm Fitt}_s(K_\p)$ and ${\rm Fitt}_r(C_\p)$, and hence ${\rm Fitt}_{r+s}(M_\p)={\rm Fitt}_s(K_\p){\rm Fitt}_r(C_\p)$. 
Since the Fitting ideals commute with localizations, we obtain 
${\rm Fitt}_{r+s}(M)_\p={\rm Fitt}_s(K)_\p{\rm Fitt}_r(C)_\p$ for every such $\p$. 
Since $R$ is an integrally closed Noetherian domain, Lemma \ref{lem} (3) yields the assertion. 
\end{proof} 

Since the polynomial ring $O[t^\Z]$ over the Dedekind domain $O=O_{F,S}$ is an integrally closed Noetherian domain, Lemma \ref{lem.dhFitt} (corrected) is applicable to the proof of \cite[Proposition]{Ueki6}.

\section*{Erratum 2}

Theorem 11.1 (1), (2) in \cite{Ueki6} ought to be stated as follows, while their consequence (3) persists. 
\setcounter{section}{11}
\setcounter{thm}{0}
\begin{thm}[corrected] \label{thm} Let the notation be as above (in Section 11.2) and suppose that $\rho:\pi\to \GL_N(O)$ is non-abelian and absolutely irreducible. Then, 

{\rm (1)} For each $n\in \N_{>0}$, the Wang exact sequence induces an injective homomorphism 
\[H_1(X_\infty,\rho)/(t^n-1)H_1(X_\infty,\rho) \inj H_1(X_n,\rho)\]
with a finite cokernel, which is isomorphic to the kernel $K_n$ of multiplication by $(t^n-1)$ on $H_0(X_\infty,\rho)$. 
The sequence $(|K_n|)_n$ is bounded. 

If ${\rm Fitt}H_0(X_\infty,\rho)=\wt{\rm Fitt}H_0(X_\infty,\rho)$ holds, then $(|K_n|)_n$ is trivial. 

{\rm (2)} Put $\Psi_n(t)={\gcd}(\ol{\Delta}_\rho(t),t^n-1)$ and $r_n={\rm Res}(\ol{\Delta}_\rho(t), (t^n-1)/\Psi_n(t))$. 
Then the equality \[\ds |{\rm tor}H_1(X_n,\rho)|=c_n k_n|r_n|\prod_{p\in S} |r_n|_p \]
holds, where $c_n$ is a bounded sequence defined by $c_n=|{\rm tor} H_1(X_\infty,\rho)/\Psi_n(t)|$ and
$k_n$ is another bounded sequence. 

If ${\rm Fitt}H_i(X_\infty,\rho)=\wt{\rm Fitt}H_i(X_\infty,\rho)$ holds for $i=0,1$, then $(k_n)_n$ is trivial. 

{\rm (3)} For any prime number $p$, the asymptotic formulas  
\[\ds\lim_{n\to \infty} |{\rm tor}H_1(X_n,\rho)|^{1/n}=\mh (\ol{\Delta}_\rho(t)) \ \ {\rm and}\ \ 
\ds\lim_{n\to \infty} ||{\rm tor}H_1(X_n,\rho)||_p^{1/n}=\mh_p (\ol{\Delta}_\rho(t))\]
of torsion growth hold. 
\end{thm}

We ought to care that $H_i(X_\infty, \rho)$ may have a non-trivial pseudo-null submodule. 
In the proof of (1), 
the multiplication by $t^n-1$ on $H_0(X_\infty,\rho)\otimes O_{(\p)}$ is not necessarily injective. 
In addition, in the proof of (2), the module $\mca{H}=H_1(X_\infty,\rho)$ does not necessarily satisfy the assumption of \cite[Theorem 3.13]{Hillman2}. 
Here we attach an additional argument to complete the proof. 

\begin{proof} 
By the structure theorem of finitely generated torsion modules over an integrally closed Noetherian domain \cite[Theorem 2.36]{Ochiai2014-Iwasawa1}, we have a homomorphism 
$\varphi_i: H_i(X_\infty,\rho)\overset{\sim}{\to} \mca{M}_i$ $=\bigoplus_j O[t^\Z]/\p_{i,j}^{r_{i,j}}$ to a standard module $\mca{M}_i$ with finite kernel and cokernel, where $(\p_{i,j})_j$ is a finite sequence of height 1 prime ideals of $O[t^\Z]$ and $r_{i,j}\in \Z_{>0}$, 
hence a finite sequence $(\mf{m}_j)_j$ of maximal ideals of $O[t^\Z]$ satisfying  ${\rm Fitt}H_i(X_\infty,\rho)=\wt{\rm Fitt}H_i(X_\infty,\rho)\prod_j \mf{m}_{i,j}$. If $\{\mf{m}_{i,j}\}_j$ is empty, then $\varphi_i$ is an isomorphism.

Consider the injective homomorphism $H_1(X_\infty,\rho)/(t^n-1)H_1(X_\infty,\rho) \inj H_1(X_n,\rho)$, whose cokernel is isomorphic to $K_n={\rm Ker}(t^n-1|_{H_0(X_\infty,\rho)})$. 
If $\{\mf{m}_{0,j}\}_j$ is empty, then the assumption on $\rho$ and \cite[Proposition A3]{FriedlKimKitayama2012} yields $|K_n|=1$. 
If otherwise, a nontrivial pseudo-null submodule of $H_0(X_\infty,\rho)$ yields a nontrivial bounded sequence $(|K_n|)_n$ contributing to $k_n$. 

If ${\rm gcd}(\ol{\Delta}_{\rho}(t),t^n-1)= 1$, then 
$t^n-1|_{\mca{M}_1}$ is injective and the sizes of the kernel and the cokernel of $t^n-1|_{{\rm Coker} \varphi_1}$ coincide. Hence by using the snake lemma twice, we obtain 
$|H_1(X_\infty,\rho)/(t^n-1)H_1(X_\infty,\rho)|$ $=|{\rm Ker}\varphi_1/(t^n-1){\rm Ker}\varphi_1||\mca{M}_1/(t^n-1)\mca{M}_1|$. 
The original proof of (2) with use of \cite[Theorem 3.13]{Hillman2} applies to $\mca{M}_1$, instead of $\mca{H}=H_1(X_\infty,\rho)$. 
If $\{\mf{m}_{1,j}\}_j$ is not empty, then the bounded sequence $(|{\rm Ker}\varphi/(t^n-1){\rm Ker}\varphi|)_n$ also contributes to $k_n$. 
This argument may be modified to the case with ${\rm gcd}(\ol{\Delta}_{\rho}(t),t^n-1)\neq 1$ as claimed at the and of the original proof. 
\end{proof} 

We remark that Theorem \ref{thm} (corrected) still implies the previous result due to R.~Tange \cite[Theorem 13]{TangeRyoto2018JKTR}. 
Indeed, if $H_0(X_\infty,\rho)=0$ holds, then the exact sequence in \cite[Proposition 4.1]{Ueki6} and 
\cite[Theorem 3.12 (2)]{Hillman2} yield that 
${\rm Fitt}H_1(X_\infty,\rho)={\rm Fitt}\mca{A}/{\rm Fitt}\mca{C}$ $=(W_\rho^\alpha(t))$. 
Thus we have 
${\rm Fitt}H_i(X_\infty,\rho)=\wt{\rm Fitt}H_i(X_\infty,\rho)$ for $i=0,1$, 
hence $k_n=1$.\\  

\section*{Acknowledgments} 
I am grateful to Jonathan Hillman, %Takahiro Kitayama, 
Hirotaka Koga, Tomoki Mihara, Yuya Murakami, Ryoto Tange, and Tomoki Yuji for helpful communication. 
This work was partially supported by JSPS KAKENHI Grant Number JP19K14538. 

\bibliographystyle{amsalpha}
%\bibliography{/Users/uekijun/Dropbox/refs1} 
\providecommand{\bysame}{\leavevmode\hbox to3em{\hrulefill}\thinspace}
\providecommand{\MR}{\relax\ifhmode\unskip\space\fi MR }
\providecommand{\MRhref}[2]{%
  \href{http://www.ams.org/mathscinet-getitem?mr=#1}{#2}
}
\providecommand{\href}[2]{#2}

\end{document}